\documentclass{amsart}
\usepackage{amsmath, amsthm, amsfonts, amssymb}
\usepackage{color}
\usepackage{float}
\usepackage{graphicx}
\usepackage{verbatim}
\usepackage{wrapfig}
\usepackage[all]{xy}
\usepackage{hyperref} 
\hypersetup{colorlinks=true, linkcolor=blue, citecolor=green, filecolor=red, urlcolor=green}

\newcommand{\N}{\mathbb{N}}                     
\newcommand{\Z}{\mathbb{Z}}                     
\newcommand{\R}{\mathbb{R}}                     
\newcommand{\D}{\mathbb{D}}
\newcommand{\C}{\mathbb{C}}
\renewcommand{\P}{\mathcal{P}}
\renewcommand{\sl}{{\rm sl}}
\newcommand{\wind}{{\rm wind}}

\numberwithin{equation}{section}
\numberwithin{newcounter}{section}
\numberwithin{figure}{section}
\numberwithin{footnote}{section}

\newtheorem{theorem}{Theorem}[section]
\newtheorem{lemma}[theorem]{Lemma}
\newtheorem{corollary}[theorem]{Corollary}
\newtheorem{proposition}[theorem]{Proposition}
\theoremstyle{definition}
\newtheorem{definition}[theorem]{Definition}

\theoremstyle{remark}
\newtheorem{remark}[theorem]{Remark}
\numberwithin{equation}{section}

\begin{document}

\title[Global sections for Reeb flows on $L(p,1)$]{Global surfaces of section for dynamically convex Reeb flows on lens spaces}

\author{A. Schneider}

\address{Universidade Estadual do Centro-Oeste, Rua Camargo Varela de S\'a, $3$, Guarapuava – PR, Brazil, 85040-080}

\email{alexsandro@unicentro.br}

\begin{abstract}
We show that a dynamically convex Reeb flow on the standard tight  lens space $(L(p, 1),\xi_{\rm std})$,  $p>1,$  admits a $p$-unknotted closed Reeb orbit $P$ which is the binding of a rational open book decomposition with disk-like pages.  Each page is a rational global surface of section for the Reeb flow and the Conley-Zehnder index of the $p$-th iterate of $P$ is $3$.
We also check dynamical convexity in the H\'enon-Heiles system for low positive energies. In this case the rational open book decomposition follows from the fact that  the sphere-like component of the energy surface admits a $\Z_3$-symmetric periodic orbit and the flow descends to a Reeb flow on the standard tight $(L(3,2),\xi_{\rm std})$.

\end{abstract}

\maketitle
\tableofcontents

\section{Introduction}

It is a classical problem in conservative dynamics to investigate the existence of periodic motions of Hamiltonian systems restricted to energy levels. Poincar\'e and Birkhoff used global surfaces of section to study periodic trajectories in the restricted three body problem and geodesic flows in the $2$-sphere. In such systems, the energy levels are diffeomorphic to the real projective $3$-space and, under some good conditions, the flow is reduced to a surface map. Tools from discrete dynamics then come into play and one can derive the existence of other periodic trajectories.

The systems considered by Poincar\'e and Birkhoff are particular cases of Reeb flows on the universally tight lens space $L(2,1) \simeq \R P^3$, as shown in \cite{albers2012contact}. In order to study such systems, one usually lifts the flow to the tight $3$-sphere via a suitable double covering map. In some cases, it is even possible to show that the Reeb flow is dynamically convex and prescribe the boundary of the global surface of section \cite{albers2012global,pedro_proceedings_icm}.

Hofer, Wysocki and Zehnder \cite{hofer1998dynamics} used  pseudoholomorphic curves in symplectizations  to study dynamically convex Reeb flows on the tight $3$-sphere. They found a special periodic orbit, which is unknotted and has Conley-Zehnder index $3$,  bounding a disk-like global surface of section.  One of the consequences of this global section and a result of Franks \cite{franks1992geodesics} is that the Reeb flow admits either $2$ or infinitely many periodic orbits. This result applies to Hamiltonian dynamics on strictly convex hypersurfaces in $\R^4$.

Using the same methods in \cite{hofer1998dynamics}, Hryniewicz and Salom\~ao \cite{hryniewicz2016elliptic} proved a similar result for Reeb flows on $\R P^3 \equiv S^3 / \Z_2 \equiv L(2,1),$ equipped with the universally tight contact structure: if the Reeb flow is dynamically convex then it admits an elliptic and $2$-unknotted periodic orbit which is the binding of a rational open book decomposition. Each page is a $2$-disk for the binding and constitutes a rational global surface of section for the Reeb flow. The main motivation of this result was to study the circular planar restricted three body problem directly on the regularized $\R P^3$, without considering the usual lift to the tight $3$-sphere.

Here we generalize results in \cite{hofer1998dynamics} and \cite{hryniewicz2016elliptic} to lens spaces  $L(p,1)$, $p>1$, equipped with the standard tight contact structure $\xi_{\rm std}$. We use the theory of pseudoholomorphic curves in symplectizations in order to show that if the Reeb flow on such a contact manifold is dynamically convex then it admits a closed Reeb orbit $P$ which is $p$-unknotted and bounds a rational disk-like global surface of section. Moreover, the Conley-Zehnder index of the $p$-th iterate of $P$ is $3$. In fact, $P$ is the binding of a rational open book decomposition whose pages are rational global surfaces of section.

Our main result applies to Hamiltonian systems whose sphere-like components of the corresponding energy levels present certain $\mathbb{Z}_{p}$-symmetry so that the flow descends to a Reeb flow on  $(L(p,1),\xi_{\rm std})$.

The  H\'enon-Heiles Hamiltonian is  $\mathbb{Z}_{3}$-symmetric and for energies $0<E<1/6$ the flow on the sphere-like component $S_E$ of its energy level descends to a Reeb flow on $(L(3,2), \xi_{\rm std})$. In particular, our previous result does not apply. In spite of that, we can use a criterium found in \cite{salomao2004convex} to check  that $S_E$ is strictly convex. Hence,  by a result of Hofer, Wysocki and Zehnder, it is also dynamically convex.  Since $S_E$ contains two $Z_3$-symmetric periodic orbits which descend to $3$-unknotted closed Reeb orbits with rational self-linking number $-1/3$, it turns out that results in \cite{hryniewicz2016elliptic} can be applied to $S_E/\Z_3$ and each one of these periodic trajectories projects to the binding of a rational open book decomposition with similar properties.

\subsection{The existence of the binding orbit} Finding necessary and sufficient conditions for a closed Reeb orbit to be the boundary of a disk-like global surface of section constitutes an important question in Reeb dynamics. The importance of this problem relies on the fact that, under some good circumstances, this closed Reeb orbit characterizes the contact manifold. Relevant research in this direction is found in \cite{hofer1996characterisation, hofer1998dynamics, hhofer1999characterization, hryniewicz2012fast, HLS}

In \cite{hofer1996characterisation,hryniewicz2012fast} the authors characterize the tight $3$-sphere. If a dynamically convex contact form  on a co-orientable contact $3$-manifold $(M,\xi)$ admits an unknotted closed Reeb orbit $P_0$ with self-linking number $-1$  then $(M,\xi)$ is contactomorphic to the tight $3$-sphere. The proof is based on the construction of an open book decomposition with binding $P_0$ whose pages are disk-like global surfaces of section for the Reeb flow. The dynamical convexity hypothesis used in the construction of the open book can be dropped and one may only require that the Reeb orbits with Conley-Zehnder index $2$ are linked to $P_0$, see \cite{hs_ontheexistenceofdisk2011}. In fact, the closed Reeb orbits with Conley-Zehnder index $2$ represent an obstruction to the existence of an open book decomposition and  other types of transverse foliations may be considered, see \cite{dePaulo_Salomao2,dePaulo_Salomao} for the existence of $3-2-3$ foliations adapted to Reeb flows on the tight $3$-sphere.

Recently, Hryniewicz, Licata and Salom\~ao \cite{HLS} characterized the universally tight lens spaces. If $(M,\xi)$ is a closed co-orientable contact manifold  admitting a Reeb flow with a $p$-unknotted closed Reeb orbit $P_0$, whose self-linking number is $-1/p$ and the Conley-Zehnder index of its $p$-th iterate is at least $3$ then a suitable necessary and sufficient condition on the closed Reeb orbits which are contractible in $M \setminus P_0$ implies that $(M,\xi)$ is a lens space $L(p,q)$, for some $1\leq  q \leq p$, equipped with the standard tight contact structure. The proof is also based on the construction of a rational open book decomposition with binding $P_0$ whose pages are rational disk-like global surfaces of section for the Reeb flow.

The results mentioned above show that the first step in the construction of a disk-like global surface of section is to show the existence of a special closed Reeb orbit which will be the binding of an open book decomposition. Hence, it is of central interest to provide sufficient conditions on the contact form which assure the existence of such a closed Reeb orbit. The following question is not answered in its full generality:

\begin{itemize}
\item[-] Does a dynamically convex Reeb flow on $(L(p,q),\xi_{\rm std})$  admit a $p$-unknotted closed Reeb orbit with self-linking number $-1/p$?
\end{itemize}

Here we give an affirmative answer to this question in the case $q=1$. Together with results in \cite{HLS} and \cite{hryniewicz2016elliptic}, we show that this closed Reeb orbit is the binding of a rational open book decomposition with disk-like pages and each page is a global surface of section.

Recall that a rational global surface of section forces the existence of a second closed Reeb orbit $P_1$ associated to a fixed point of the first return map. The link formed by $P_0$ and $P_1$ is called a Hopf link. A non-resonance condition on the rotation numbers of $P_0$ and $P_1$ corresponds to a twist condition of the first return map and forces the existence of infinitely many closed Reeb orbits, see also \cite{HMS}. Alternatively, a result of Franks \cite{franks1992geodesics} gives infinitely many periodic orbits in case a third one exists. In particular, all such Reeb flows admitting disk-like global surfaces of section have $2$ or infinitely many closed Reeb orbits. It is still an open question whether every $3$-dimensional Reeb flow admits either $2$ or infinitely many closed Reeb orbits. See \cite{ghptorsioncontact2017} for partial answers to this question.

\subsection{Basic concepts and main result} Let $M$ be a smooth oriented $3$-manifold. A contact structure on $M$ is a smooth hyperplane distribution $\xi\subset TM$, locally defined by the kernel of a $1$-form $\lambda$ satisfying $\lambda\wedge d\lambda\neq0$. The pair $(M, \xi)$ is a contact manifold. We say that $(M,\xi)$ is co-orientable if there exists a globally defined $1$-form $\lambda$ on $M$ satisfying $\xi=\ker \lambda$. If $f:M \to \R \setminus \{0\}$ is smooth then $\lambda$ and $f\lambda$ define the same contact structure. Furthermore, since $(f\lambda)\wedge d(f\lambda)=f^2 \lambda \wedge d\lambda$, any contact form defining $\xi$ induces the same orientation on $M$ and we say that $\xi$ is positive if the induced orientation coincides with the orientation of $M$.

A contact structure $\xi$ on $M$ is called overtwisted if there exists an embedded disk $D\hookrightarrow M$ such that $T_{z}(\partial D)\subset\xi_{z}$ and $T_{z}D\neq\xi_{z}$ for all $z\in\partial D$. If such a disk does not exist then the contact structure is called tight.
Let $\lambda$ be a contact form which defines the contact structure $\xi$ on $M$. The vector field on $M$, uniquely determined by
\begin{equation*}
\imath_{X_{\lambda}}d\lambda=0 \mbox{ and } \imath_{X_{\lambda}}\lambda=1,
\end{equation*}
is called the Reeb vector field of $\lambda$. Its flow $\{\varphi_t,t\in \R\}$ is the Reeb flow of $\lambda$. Let $P=(x, T)$ be a periodic orbit of the Reeb flow of $\lambda$, that is  $x:\R \to M$ is periodic, it satisfies $x(t)=\varphi_t(x(0)) \forall t,$ and $T>0$ is a period of $x$. $P$ is also called a closed Reeb orbit of $\lambda$. We say that $P$ is simple if $T$ is the least positive period of $x$. We denote by $\mathcal{P}(\lambda)$ the set of equivalence classes of periodic orbits of $\lambda$ with the identification
\begin{equation*}
P=(x, T)\sim Q=(y, T') \Leftrightarrow  T=T' \mbox{ and } x(\R)=y(\R).
\end{equation*}
 We say that $P=(x, T)$ is contractible if the loop
\begin{equation}\label{loop}
x_{T}:\R / \Z \to M: t\in\mathbb{R}/\mathbb{Z}\mapsto x(Tt),
\end{equation}
is contractible on $M$. If the first Chern class $c_{1}(\xi)$ vanishes on $\pi_{2}(M)$ then every contractible periodic orbit $P\in\mathcal{P}(\lambda)$ has a well-defined Conley-Zehnder index $\mu_{CZ}(P)\in\mathbb{Z}$.

The dynamical convexity condition was introduced in \cite{hofer1998dynamics}.

\begin{definition}[Hofer, Wysocki and Zehnder]
A contact form $\lambda$ on a smooth closed $3$-manifold $M$ is called  dynamically convex if $c_{1}(\mathrm{ker}\lambda)$ vanishes on $\pi_{2}(M)$ and every contractible periodic orbit $P\in\mathcal{P}(\lambda)$ satisfies $\mu_{CZ}(P)\geq 3$.
\end{definition}

Dynamical convexity imposes obstructions on contact manifolds.

\begin{theorem}[Hofer, Wysocki and Zehnder \cite{hhofer1999characterization}]\label{teo:dynamical_tight}
If $\lambda$ is a dynamically convex contact form on a closed $3$-manifold $M$ then $\pi_{2}(M)$ vanishes and the contact structure $\ker \lambda$ is tight.
\end{theorem}

A knot $K\hookrightarrow M$ is called $k$-unknotted, for some $k\in\mathbb{N}$, if there exists an immersion $u:\mathbb{D}\rightarrow M$, so that $u|_{\D \setminus \partial \D}$ is an embedding and $u|_{\partial\mathbb{D}}:\partial\mathbb{D}\rightarrow K$ is a $k$-covering map. The map $u$ is called a \textit{$k$-disk} for $K$. If $K$ is oriented then we say that $u$ induces the same orientation as $K$ if $u|_{\partial\mathbb{D}}$ preserves orientation, where $\partial\mathbb{D}$  has the counter-clockwise orientation. If the $k$-unknotted $K$ is transverse to $\xi$ then $K$ is oriented by $\lambda$ and there exists a well-defined rational self-linking number $\sl(K,u)\in\mathbb{Q}$, computed with respect to a $k$-disk $u$ for $K$, see \cite{baker2012rational}. If the first Chern class $c_{1}(\xi)$ vanishes on $\pi_{2}(M)$ then $\sl(K)=\sl(K,u)$ does not depend on the choice of $u$.

\begin{definition}
Let $\lambda$ be a defining contact form for a closed contact $3$-manifold $(M,\xi)$. Let $K \hookrightarrow M$ be a $k$-unknotted closed Reeb orbit of $\lambda$. A rational disk-like global surface of section bounded by $K$ is a $k$-disk $u:\mathbb{D}\rightarrow M$ for $K=u(\partial\mathbb{D})$ so that $u(\mathbb{D}\setminus\partial\mathbb{D})$ is transverse to $X_{\lambda}$, and every Reeb trajectory in $M\setminus K$ hits $u(\D \setminus \partial \D)$ infinitely many times in the past and in the future. In particular, the Reeb flow of $\lambda$ is encoded in the corresponding first return map $\psi: u(\D \setminus \partial \D) \to u(\D \setminus \partial \D)$.
\end{definition}

Let $(x_1,x_2,y_1,y_2)$ be coordinates in $\R^4$. Equip $\R^4$ with the standard symplectic form
$$
\omega_0 = \sum_{i=1}^2 dx_i \wedge dy_i.
$$
The Liouville form
\begin{equation*}
\lambda_{0}=\frac{1}{2}\sum_{i=1}^{2}(x_i\,dy_i-y_i\,dx_i),
\end{equation*}
is a primitive of $\omega_0$ and restricts to a contact form on the $3$-sphere
\begin{equation*}
S^3=\{x_1^2+x_2^2+y_1^2+y_2^2=1\}.
\end{equation*}
The contact structure $\xi_{\rm std}=\ker \lambda_0$ is called the standard tight contact structure on $S^3$. It is well known that it is the unique tight contact structure on $S^3$ up to diffeomorphism.

Given relatively prime integers $p \geq q \geq 1$, there exists a free action of $\mathbb{Z}_{p}:=\mathbb{Z}/p\mathbb{Z}$ on $(S^{3}, \xi_{0})$ induced by $g_{p, q}:\C^2 \to \C^2$,
\begin{equation}\label{contactomorphism}
g_{p, q}(z_1=x_1+iy_1, z_2=x_2+iy_2)=\left(e^{2\pi i/p}z_1, e^{2\pi iq/p}z_2\right),
\end{equation} via the identification $\C^2 \equiv \R^4$.
The orbit space
\begin{equation*}
L(p, q):=S^3/\mathbb{Z}_{p},
\end{equation*}
is called a lens space. Since $g_{p,q}^*\lambda_0 = \lambda_0$ both  $\lambda_0$ and $\xi_{\rm std}$ descend to $L(p,q)$. They will be referred to as the Liouville form and the standard tight contact structure on $L(p,q)$, respectively, and will be still denoted by $\lambda_0$ and $\xi_{\rm std}$.

\begin{definition}
Let $\lambda$ be a defining contact form on $(L(p,q),\xi_{\rm std})$.  A rational open book decomposition with disk-like pages and binding orbit $K$ is a pair $(\pi,K)$ formed by a $p$-unknotted closed Reeb orbit $K\hookrightarrow L(p,q)$ and a smooth fibration $\pi:M \setminus K \to S^1$ so that the closure of each fiber $\pi^{-1}(t)$ is the image of a $k$-disk for $K$.
\end{definition}

Our main statement extends the main results in \cite{hofer1998dynamics,hryniewicz2016elliptic} for $(L(p,1),\xi_{\rm std})$.

\begin{theorem}\label{teo:lp}
 Let $\lambda$ be a  defining contact form on $(L(p, 1),\xi_{\rm std})$. If $\lambda$ is dynamically convex then its Reeb flow admits a $p$-unknotted closed Reeb orbit which is the binding of a rational open book decomposition. Each page of the open book is a rational disk-like global surface of section. Moreover, the Conley-Zehnder index of the $p$-th iterate of $P$ is $3$.
\end{theorem}

\begin{remark} A similar result is expected to hold on $(L(p,q),\xi_{\rm std})$ for $q\neq 1$. However, the Conley-Zehnder index of the $p$-th iterate of the binding orbit $P$ might not necessarily be equal equal to $3$.
\end{remark}

See \cite{frauenfelder2016real} for the existence of other types of $\Z_2$-symmetric disk-like global surfaces of section for Reeb flows on the tight $3$-sphere.

\subsection{An application}A Hamiltonian function $H:\mathbb{R}^{4}\rightarrow\mathbb{R}$ determines the vector field $X_{H}$ by $$\imath_{X_{H}}\omega_0=-dH.$$
Its flow preserves each energy level $S_{E}=H^{-1}(E)$ and we say that $S_E$ is starshaped with respect to $0\in \R^4$ if $S_E\subset \R^4$ is an embedded $3$-sphere, the origin $0$  is contained in the bounded component of $\R^4 \setminus S_E$ and each ray issuing from $0$ intersects $S_E$ transversally at a unique point. In this case $\lambda:=\lambda_0|_{S_E}$ is a contact form on $S_E$ and its Reeb flow is equivalent to the Hamiltonian flow of $H$ on $S_E$.

Let us assume that $H$ is $\mathbb{Z}_{p}$-symmetric with respect to the $\Z_p$-action given in \eqref{contactomorphism}. Notice that the Liouville form $\lambda_0$ is also $\Z_p$-symmetric. It follows that the Reeb flow of $\lambda$ descends to a Reeb flow on  $S_{E}/\mathbb{Z}_{p}\simeq (L(p, 1),\xi_{\rm std})$. The following theorem is a corollary of Theorem \ref{teo:lp}.

\begin{theorem}\label{teo:aplic_lp}
Assume that $\lambda=\lambda_0|_{S_E}$ is a dynamically convex contact form on the starshaped hypersurface $S_E=H^{-1}(E)\subset \R^4$. Assume, moreover, that $H$ is $\Z_p$-symmetric with respect to the action generated by $g_{p,1}$, see \eqref{contactomorphism}. Then the Reeb flow of $\lambda$ admits a $\Z_p$-sym\-metric unknotted periodic orbit $P$ with Conley-Zehnder index $3$ which is the binding of an open book decomposition. Each page of the open book is a disk-like global surfaces of section. Moreover, this open book is the lift to $S_E$ of a rational open book decomposition adapted to the projected Reeb flow on $S_E/\Z_p \equiv ( L(p,1),\xi_{\rm std})$.
\end{theorem}

\subsection{An example}Consider the decoupled Hamiltonian
\begin{equation*}
H=\frac{x_{2}^{2}+y_{2}^{2}}{2} + \frac{x_{1}^{2}+y_{1}^{2}}{2}+2(x_{1}^{2}+y_{1}^{2})(y_{1}x_1-
 x_{1}y_1)-4(x_{1}^{6}-3x_{1}^{4}y_{1}^{2}-3x_{1}^{2}y_{1}^{4}+y_{1}^{6}).
\end{equation*}
The origin $0\in \R^4$ is a nondegenerate local minimum of $H$. Hence for every $E>0$ sufficiently small, $H^{-1}(E)$ contains a strictly convex sphere-like subset $S_E$ close to $0\in \R^4$.

The Hamiltonian $H$ is invariant under the $\Z_4$-action generated by $g_{4,1}$, see \eqref{contactomorphism}.
Since $g_{4,1}^* \lambda_0 = \lambda_0$ the Hamiltonian flow on $S_E$ descends to a dynamically convex Reeb flow on $(L(4,1), \mathcal{\xi}_{\rm std})$. Theorem \ref{teo:aplic_lp} implies the existence of a rational open book decomposition adapted to the Reeb flow on $S_E / \Z_4$.

\begin{figure}
    \includegraphics[width=0.6\textwidth]{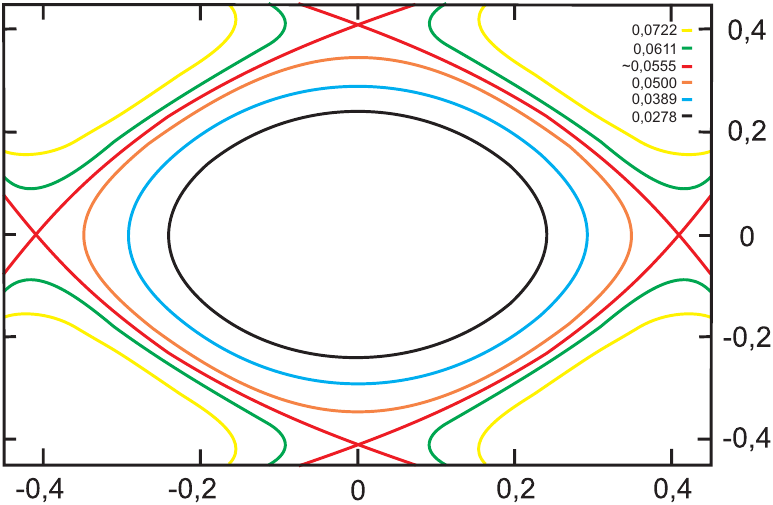}
    \caption{Projections of energy levels of $H$ to the $x_1y_1$-plane.}
    \label{fig:hill_proj}
\end{figure}

\section{The H\'enon-Heiles Hamiltonian}In 1964 H\'enon and Heiles \cite{henon1964applicability} studied the motion of a star in a galactic center with an axes of symmetry. They considered the following Hamiltonian
\begin{equation}\label{eq:henon}
H(x_{1}, x_{2}, y_{1}, y_{2})=\frac{y_{1}^{2}+y_{2}^{2}}{2}+V(x_{1}, x_{2}),
\end{equation}
where the  potential is the degree $3$ polynomial function
\begin{equation*}
V(x_1,x_2)=\frac{x_{1}^{2}+x_{2}^{2}}{2}+x_{1}^{2}x_{2}-\frac{x_{2}^ {3}}{3}.
\end{equation*}
Despite the simplicity of this Hamiltonian, its dynamics is very complex, having chaotic behavior and even multiple horseshoes  \cite{arioli2001symbolic,churchill1980pathology,ragazzo1994nonintegrability}. 

The Hamiltonian $H$ is invariant under the $\Z_3$-action on $\R^4$ generated by $\widehat g_{3,1}:\C^2 \to \C^2$
\begin{equation*}
(w_1:=x_1+ix_2, w_2:=y_1+iy_2) \mapsto \widehat  g_{3,1}(w_1,w_2)=\left(e^{\frac{2 \pi i}{3}}w_1,e^{\frac{2 \pi i}{3}}w_2 \right).
\end{equation*} Notice that $\widehat g_{3,1}$ differs from $g_{3,1}$ defined as in \eqref{contactomorphism}. Indeed  each $\C$-factor in the above definition is a Lagrangian subspace.
Despite this difference to $g_{3,1}$, the map $\widehat g_{3,1}$ also preserves $\lambda_0$. Although the quotient space $S^3/\Z_3$ is diffeomorphic to $L(3,1)$, the induced contact manifold $(S^3/\Z_3,\widehat \xi)$  is not contactomorphic to $(L(3,1),\xi_{\rm std})$. Indeed, one easily verifies that the orthogonal map $\psi:\R^4 \to \R^4$ defined by
$$
\psi:(x_1,x_2,y_1,y_2)= \frac{1}{\sqrt{2}}(x_1-y_2,-x_1-y_2,x_2+y_1,x_2-y_1),
$$ satisfies 
$$
\psi(S^3)=S^3, \ \ \ \psi^* \lambda_0 = \lambda_0 \  \ \mbox{ and } \ \ \psi \circ \widehat g_{3,1} = g_{3,2} \circ \psi.
$$Therefore,  $(S^3/\Z_3,\widehat \xi)$ is contactomorphic to $(L(3,2),\xi_{\rm std})$. 

For every $0< E <\frac{1}{6}$, the energy level $H^{-1}(E)$ contains a sphere-like component $S_E$. As we shall see below, $S_E$ is strictly convex  and hence the Hamiltonian flow on $S_E$ descends to a dynamically convex Reeb flow on $(L(3,2),\xi_{\rm std})$. Although  Theorem \ref{teo:aplic_lp} does not apply to the H\'enon-Heiles system, the rational open book decomposition on $S_{E}/\mathbb{Z}_{3}$ still exists as we shall explain below.

First we show that $S_E$ is strictly convex and, in particular, it is dynamically convex. In order to prove this fact we use a simple criterium found in \cite{salomao2004convex}.

\begin{theorem}[Salom\~ao \cite{salomao2004convex}] \label{teo:curv_positiva}Let $H=\frac{y_1^2+y_2^2}{2}+V(x_1,x_2)$ be a Hamiltonian function on $\R^4$ with smooth potential $V:\R^2 \to \R$.
Let $S_E\subset H^{-1}(E)$ be a sphere-like regular component of its energy level. Denote by $\pi:\R^4 \to \R^2$ the natural projection to the $x_1x_2$-plane.  Let $B_{E}=\pi(S_E)$. Then $S_{E}$ is strictly convex if and only if the function
\begin{equation*}
G_{E}:=2(E-V)(V_{x_{1}x_{1}}V_{x_{2}x_{2}}-V_{x_{1}x_{2}}^{2})+V_{x_{1}x_{1}}V_{x_{2}}^{2}+
V_{x_{2}x_{2}}V_{x_{1}}^{2}-2V_{x_{1}}V_{x_{2}}V_{x_{1}x_{2}}
\end{equation*}
is positive on $B_{E}$.
\end{theorem}

\begin{figure}
    \includegraphics[width=0.6\textwidth]{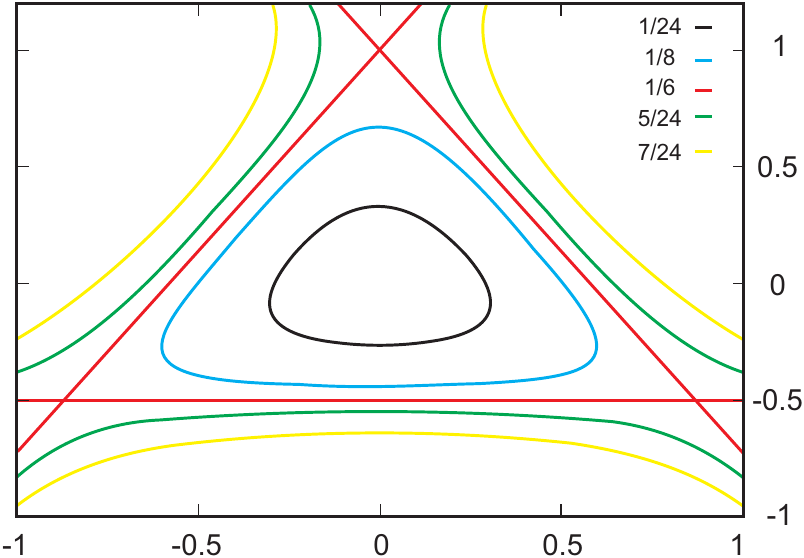}
    \caption{The boundary of Hill's region of the H\'enon-Heiles Hamiltonian for various values of energy.}
    \label{fig:hh_proj}
\end{figure}

We use Theorem \ref{teo:curv_positiva} to prove that $S_E$ is strictly convex.

\begin{theorem}\label{teo:henon_dinami_convexo}
For every $0<E<\frac{1}{6}$, $S_E$ is a strictly convex hypersurface. In particular, the Hamiltonian flow on $S_E$ is dynamically convex.
\end{theorem}
\begin{proof}
A direct computation shows that
\begin{equation*}
G_{E}=2(E-V)\left(1-4x_{1}^{2}-4x_{2}^{2}\right)+I,
\end{equation*}
where
\begin{equation}\label{eq:bordo}
I:=x_{2}^{2}-3x_{2}^{4}+2x_{2}^{5}+x_{1}^{2}\left(1-6x_{2}^{2}-4x_{2}^{3}+
x_{1}^{2}(-3-6x_{2})\right).
\end{equation}

For every $0<E<\frac{1}{6}$,  $B_{E}$ is contained in the interior of the triangle $B_{\frac{1}{6}}:=\pi(S_{\frac{1}{6}})$,  whose vertices are the points
\begin{equation*}
\left(-\frac{\sqrt{3}}{2},\  -\frac{1}{2}\right), \ (0, 1) \ \mbox{ and } \ \left(\frac{\sqrt{3}}{2}, -\frac{1}{2}\right),
\end{equation*}
see Figure \ref{fig:hh_proj}. In particular, every point $(x_1,x_2) \in B_E$ satisfies
\begin{equation}\label{circ}
x_1^2 + x_2^2 < 1.
\end{equation}
Since $6V = 3x_1^2+3x_2^2 + 6x_1^2x_2 -2 x_2^3$, we have
\begin{equation}\label{b1}
x_1^2 (3+6x_2)=6V-3x_2^2+2x_2^3.
\end{equation}
This implies
\begin{equation*}
\begin{aligned}
I & = x_2^2-3x_2^4+2x_2^5 +x_1^2(1-6x_2^2-4x_2^3 - 6V+3x_2^2-2x_2^3)\\
& = x_2^2-3x_2^4+2x_2^5 +x_1^2(1-6V-3x_2^2-6x_2^3)\\
& = x_2^2 -3x_2^4+2x_2^5+x_1^2(1-6V) -6Vx_2^2+3x_2^4-2x_ 2^5\\
& = (1-6V)(x_1^2+x_2^2).
\end{aligned}
\end{equation*}
Using \eqref{circ} and the inequality $1>6E$ we obtain
\begin{equation*}
\begin{aligned}
G_E & =2(E-V)(1-4x_1^2-4x_2^2) +(1-6V)(x_1^2+x_2^2)\\
& \geq  2(E-V)(1-4x_1^2-4x_2^2) +(6E-6V)(x_1^2+x_2^2)\\
& = 2(E-V)(1-x_1^2-x_2^2)\\
& \geq 0.
\end{aligned}
\end{equation*}
Observe that the first inequality above is strict if $(x_1,x_2)\neq (0,0)$. Since $V(0,0)=0 < E$, the last inequality above is strict if $(x_1,x_2)=(0,0)$. We conclude that $G_E>0$ on $B_E$, as desired. An application of Theorem \ref{teo:curv_positiva} finishes the proof of this theorem.
\end{proof}

Theorem \ref{teo:henon_dinami_convexo} implies that the Hamiltonian flow restricted to $S_{E}$ is dynamically convex and descends to a dynamically convex Reeb flow on $S_{E}/\mathbb{Z}_{3}\equiv (L(3, 2), \mathcal{\xi}_{\rm std})$.

In \cite{churchill1979survey}, Churchill, Pecelli and Rod explain how one can use a shooting argument to obtain at least $8$ periodic orbits on $S_E$, denoted $\Pi_i,i=1\ldots8$. See Figure \ref{fig:hh_proj_simet} for their projections to the $x_1x_2$-plane. The closed orbits $\Pi_7$ and $\Pi_8$ are $\Z_3$-symmetric and project to simple closed curves on the $x_1x_2$-plane. Hence both descend to $3$-unknotted closed Reeb orbits $\widehat \Pi_7$ and $\widehat \Pi_8$ on $L(3,2)$. Moreover, their rational self-linking numbers are $-1/3$.

\begin{figure}
    \includegraphics[width=0.5\textwidth]{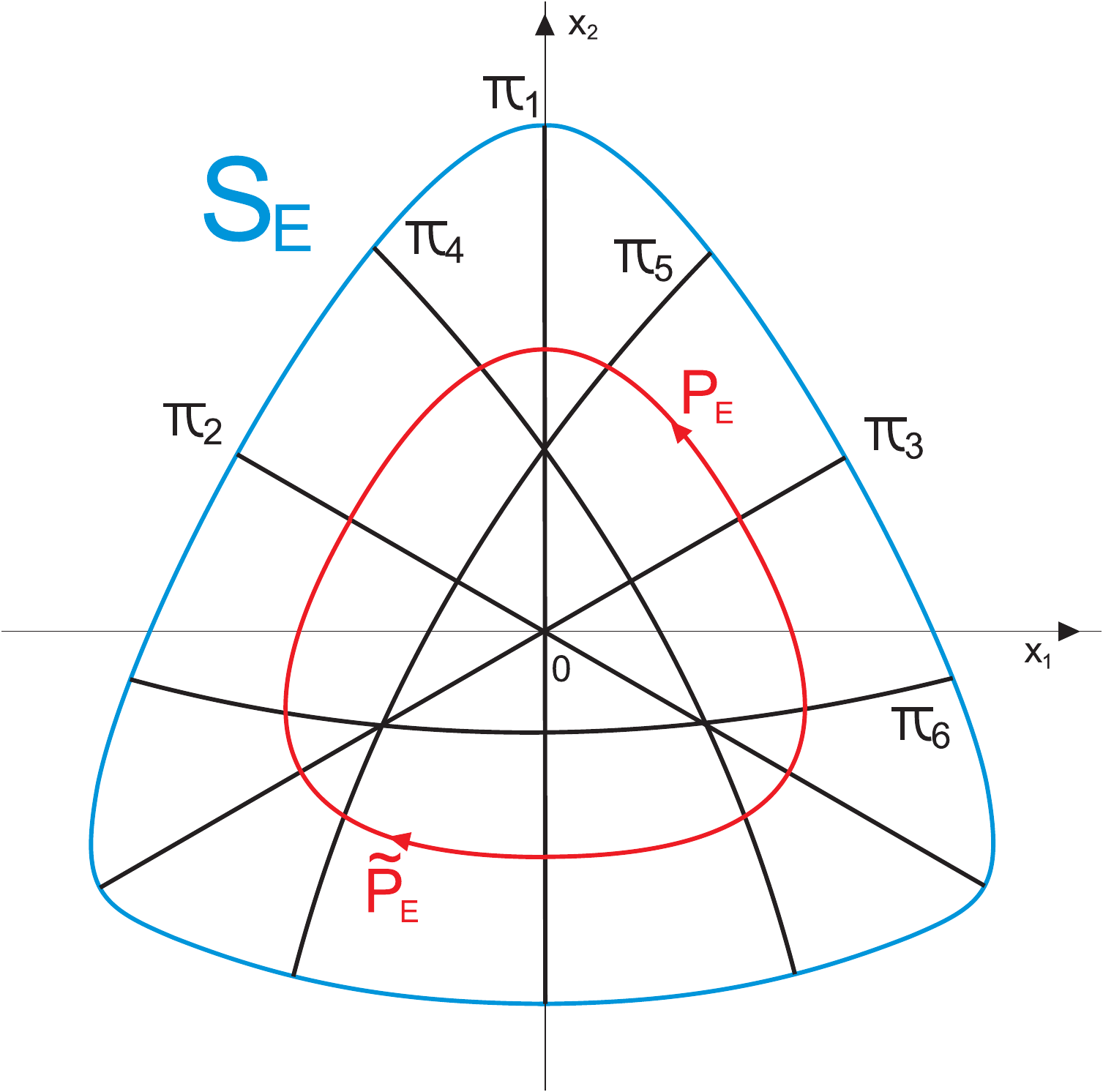}
    \caption{The projections of the $8$ periodic orbits on $S_{E}$.}
    \label{fig:hh_proj_simet}
\end{figure}

The following theorem proved by Hryniewicz and Salom\~ao guarantees that each $\widehat \Pi_i, i=7,8$ is the binding of a rational open book decomposition whose pages are rational disk-like global surfaces of section for the Reeb flow on $S_E/\Z_3 \equiv (L(3,2),\xi_{\rm std})$.

\begin{theorem}[Corollary 1.8, \cite{hryniewicz2016elliptic}] \label{theo:existence_p_disk}
If $\lambda$ is a dynamically convex contact form on $(L(p,q),\xi_{\rm std})$ then every $p$-unknotted closed Reeb orbit $K$ of $\lambda$, which has rational self-linking number $-1/p$, is the bound of a $p$-disk which is a rational global surface of section for the Reeb flow. Moreover, this $p$-disk is a page of a rational open book decomposition of $L(p, q)$ with binding $K$ such that every page is a rational disk-like global surfaces of section.
\end{theorem}

The following theorem is a direct consequence of Theorems \ref{teo:henon_dinami_convexo} and \ref{theo:existence_p_disk}.

\begin{theorem}\label{thm:2_orbitas}
For every $0<E<\frac{1}{6}$, the Reeb flow on $S_E/\Z_3 \equiv (L(3,2),\xi_{\rm std})$ admits a rational open book decomposition with binding $\widehat \Pi_7$ and whose pages a rational disk-like global surfaces of section. A similar statement holds for $\widehat \Pi_8$.
\end{theorem}

Combining Theorem \ref{thm:2_orbitas} and a result of J. Franks in \cite{franks1992geodesics} we obtain the following corollary.

\begin{corollary}
For every $0<E<\frac{1}{6}$, $S_E$ admits infinitely many periodic orbits.
\end{corollary}

\noindent\textbf{Acknowledgements.} I would like to thank my PhD advisor Pedro Salom\~ao for proposing this problem to me and for all his support and motivation. Without his help this work would probably not be done. I also thank the anonymous referee for pointing out a mistake in the application of our results to the H\'enon-Heiles system. ChatGPT was used as an aid in reviewing and revising the proof. The author takes full responsibility for all mathematical arguments and conclusions. AS was partially supported by CAPES grant 1526852 and CNPq grant 142059/2016-1.

\section{Preliminaries}\label{preliminares}
In this section we remind some definitions in contact and symplectic geometry such as Conley-Zehnder index, transverse rotation number and self-linking number. After that, we introduce pseudo-holomorphic curves in symplectizations and some related algebraic invariants.

\subsection{The Conley-Zehnder index}
Let $S:\R / \Z\rightarrow\mathbb{R}^{2\times 2}$ be a smooth path of $2 \times 2$ symmetric matrices. Identify $\R^{2}\simeq \C$, $\mathbb{R}^{2\times 2}\simeq\mathcal{L}_{\mathbb{R}}(\mathbb{C})$ and consider the unbounded self-adjoint operator
\begin{equation}\label{op:auto_adjunto}
L_{S}=-i\partial_{t}-S(t),
\end{equation}
defined on the Hilbert space $L^{2}(\mathbb{R}/\mathbb{Z}, \mathbb{C})$, with the structure induced by the euclidian inner product in $\mathbb{C}$. In \cite{properties_2}, Hofer, Wysocky and Zehnder present some important properties of the operator (\ref{op:auto_adjunto}). Its spectrum $\sigma(L_{S})$ consists of countably many real eigenvalues which accumulate only at $\pm\infty$.

For each $\eta\in\sigma(L_{S})$,  every non-trivial $\eta$-eigensection $e:\mathbb{R}/\mathbb{Z}\rightarrow\mathbb{C}$  never vanishes and so it has a well-defined  winding number
\begin{equation*}
\mathrm{wind}(\eta)=\frac{1}{2\pi}(\theta(1)-\theta(0))\in\mathbb{Z},
\end{equation*}
where $\theta:[0, 1]\rightarrow\mathbb{R}$ is any continuous function satisfying $e(t)\in\mathbb{R}^{+}e^{i\theta(t)}$. It can be shown that ${\rm wind}(\eta)$ does not depend on $e$. If $\eta_{1}\leq\eta_{2}\in \sigma(L_S)$, then $\mathrm{wind}(\eta_{1})\leq\mathrm{wind}(\eta_{2})$, and for each $k\in\mathbb{Z}$, there exist precisely two eigenvalues (multiplicities counted) whose winding number is $k$.

Let ${\rm Sp}(2)$ be the group of $2\times 2$ symplectic matrices and consider the set $\Sigma$ of smooth maps $\varphi:\mathbb{R}\rightarrow {\rm Sp}(2)$ starting at the identity and satisfying
\begin{equation*}
\varphi(t+1)=\varphi(t)\varphi(1), \ \  \forall t.
\end{equation*}
It follows that $S(t):=-i\varphi'(t)\varphi(t)^{-1}$ is a smooth path of $1$-periodic symmetric matrices and one can associate to the path $\varphi$ an unbounded self-adjoint operator $L_{S}$ as in (\ref{op:auto_adjunto}).

Consider the special eigenvalues
\begin{equation*}
\begin{aligned}
\eta^{<0}:=\mathrm{max}\{\eta\in\sigma(L_{S})|\eta<0\},\\
\eta^{\geq0}:=\mathrm{min}\{\eta\in\sigma(L_{S})|\eta\geq0\}.
\end{aligned}
\end{equation*}

\begin{definition}
The \textit{Conley-Zehnder index} of the path $\varphi$ is defined as
\begin{equation*}
\mu_{CZ}(\varphi):=\mathrm{wind}(\eta^{<0})+\mathrm{wind}(\eta^{\geq0}).
\end{equation*}
\end{definition}

Let $\Pi: E\rightarrow\mathbb{R}/\mathbb{Z}$ be an oriented real vector bundle satisfying $\mathrm{rank}_{\mathbb{R}}(E)=2$. Denote by $\Omega_{E}^{+}$ the set of homotopy classes of oriented trivializations of $E$. For any continuous non-vanishing section $t\in\mathbb{R}/\mathbb{Z}\mapsto Z(t)\in\Pi^{-1}(t)$, there exists another continuous non-vanishing section $Z'$, such that $\{Z(t), Z'(t)\}$ is an oriented basis of $\Pi^{-1}(t) \forall t$. The frame $\{Z, Z'\}$ determines an oriented trivialization and its homotopy class $\beta\in\Omega_{E}^{+}$ depends only on $Z$, up to homotopy through non-vanishing sections. It is called the homotopy class induced by $Z$, see \cite{HLS}.
Given a continuous non-vanishing section $W$ of $E$, it follows that $W(t)=a(t)Z(t)+b(t)Z'(t)$ for unique continuous  functions $a$, $b$ and
\begin{equation}\label{preliminaries_wind}
\mathrm{wind}(W, Z):=\frac{1}{2\pi}(\theta(1)-\theta(0))\in\mathbb{Z},
\end{equation}
where $\theta:[0, 1]\rightarrow\mathbb{R}$ is a continuous function satisfying $a(t)+ib(t)\in\mathbb{R}^{+}e^{i\theta(t)}$. This integer depends only on the homotopy classes of non-vanishing sections $Z$ and $W$. Denoting by $\beta'\in\Omega_{E}^{+}$ the homotopy class of oriented trivializations induced by $W$, we may write $\mathrm{wind}(\beta', \beta)=\mathrm{wind}(W, Z)$.

Now let $P=(x, T)\in \P(\lambda)$ be a closed Reeb orbit of the Reeb flow of a contact $\lambda$ on a closed $3$-manifold $M$. The contact structure $\xi=\ker \lambda$ is preserved by the flow. Consider the map $x_T:\R / \Z \to M$ defined by $x_T:=x(T\cdot)$. Then the bundle $x_{T}^{*}\xi\rightarrow\mathbb{R}/\mathbb{Z}$ is oriented by $d\lambda$. Choose a $d\lambda$-symplectic trivialization $\Psi:x_{T}^{*}\xi\rightarrow\mathbb{R}/\mathbb{Z}\times\mathbb{R}^{2}$ representing a class $\beta\in\Omega_{x_{T}^{*}\xi}^{+}$. Define the path of symplectic matrices $\varphi\in\Sigma$ by
\begin{equation}\label{caminho_matriz_simpletica}
\varphi(t):=\Psi_{t}\circ d\phi_{Tt}\circ \Psi_{0}^{-1},
\end{equation}
where $\Psi_{t}$ is the restriction of $\Psi$ to the fiber over $t$. The \textit{Conley-Zehnder index} of $(P, \beta)$ is  defined as
\begin{equation}\label{ind_conley_zehnder}
\mu_{CZ}(P, \beta):=\mu_{CZ}(\varphi) \in \Z.
\end{equation}
It does not depend on the choice of $\Psi$ in class $\beta$. Moreover,
\begin{equation}\label{relacao_conley_bases}
\mu_{CZ}(P, \beta)=\mu_{CZ}(P, \beta')+2\mathrm{wind}(\beta', \beta),
\end{equation} for every $\beta, \beta' \in \Omega^+_{x_T^*\xi}$.

If $P\subset M$ is contractible  and $c_{1}(\xi)|_{\pi_{2}(M)}\equiv 0$  then there exists a special class $\beta_{\rm disk}\in\Omega_{x_{T}^{*}\xi}^{+}$ induced by any trivialization of $x_T^*\xi$ which extends to a trivialization of $u^*\xi$, where $u:\D \to M$ is a capping disk for $P$, i.e., $u$ is continuous and $u(e^{2 \pi i t}) =x_T(t) \forall t$.  We may denote $\mu_{CZ}(P,\beta_{\rm disk})$ simply by $\mu_{CZ}(P)$.

Let $\varphi$ be the path of $2\times 2$ symplectic matrices defined as in (\ref{caminho_matriz_simpletica}), which is associated to the closed Reeb orbit $P\in \P(\lambda)$ and the homotopy class $\beta\in \Omega^+_{x_T^*\xi}$. In \cite{hofer2003finite} the authors present a more geometrical definition of  $\mu_{CZ}(P)$. Given  $0\neq v\in\mathbb{R}^{2}$, choose a continuous argument $\theta_{v}(t)$ of $\varphi_t v$. Let $\Delta(v):=(\theta_{v}(1)-\theta_{v}(0))/2\pi$ and consider the interval $J:=\{\Delta(v):0\neq v\in\mathbb{R}^{2}\}$.  It is possible to show that $J$  has length $<1/2$. Moreover, for every $\epsilon>0$ sufficiently small, we have
\begin{equation}\label{indice_geometrico}
\mu_{CZ}(P, \beta)= \left\{
\begin{array}{ll}
\displaystyle 2k,\hspace{0.2cm}\mathrm{if}\hspace{0.2cm}k\in J_{\epsilon} \\
\displaystyle 2k-1,\hspace{0.2cm}\mathrm{if}\hspace{0.2cm}J_{\epsilon}\subset(k-1, k),
\end{array}
\right.
\end{equation}
where $J_\epsilon:=J - \epsilon$.

The transverse rotation number of $P$ is defined as
\begin{equation*}
\rho(P, \beta):=\lim_{t\rightarrow\infty}\frac{\theta_{v}(t)}{2\pi t}.
\end{equation*}
The limit above does not depend on $v$. It satisfies
\begin{equation*}
\rho(P, \beta)=\rho(P,\beta')+{\rm wind}(\beta',\beta),
\end{equation*}
for every $\beta, \beta' \in\Omega^+_{x_T^*\xi}$.

Given $n\in \N$, denote by $P^{n}:=(x, nT)$ the $n$-th iterate of $P$. It follows that
\begin{equation}\label{eq:numero_rot}
\rho(P^{n}, \beta^{n})=n\rho(P, \beta),
\end{equation}
where $\beta^{n}\in\Omega_{x_{nT}^{*}\xi}^{+}$ is the homotopy class induced by $\beta\in\Omega_{x_{T}^{*}\xi}^{+}$. If  $P^{n}$ is non-degenerate, then the following relation is verified	
\begin{equation}\label{relacao_conley_rotacao}
\mu_{CZ}(P^{n}, \beta^{n})=\left\{
\begin{array}{ll}
 2n \rho(P, \beta), & \mbox{ if }P^{n}\mbox{ is }\mbox{ hyperbolic.} \\
 2\lfloor n\rho(P, \beta)\rfloor+1, & \mbox{ if }P^{n} \mbox{ is }\mbox{ elliptic.}
\end{array}
\right.
\end{equation}
Hence $\mu_{CZ}(P,\beta)=2$ if and only if $P$ is hyperbolic and $\rho(P,\beta)=1$.

A $d\lambda$-compatible complex structure $J$  on $x_{T}^* \xi$ determines the inner product
\begin{equation*}
\int_{\mathbb{R}/\mathbb{Z}}d\lambda_{x_{T}(t)}(Z(t), J_{t}\cdot W(t))dt, \ \ \ \forall Z,W \in L^2(x_{T}^{*}\xi).
\end{equation*}
Here $J_{t}$ is the restriction of $J$ to the fiber over $t$. Using an auxiliary symmetric connection $\nabla$ on $TM$ we define the so called asymptotic operator $A_P$ as the unbounded self-adjoint operator
\begin{equation*}
A_{P}:\eta\mapsto-J_{t} \cdot (\nabla_{t}\eta-T\nabla_{\eta}X_{\lambda}),
\end{equation*}
defined on $W^{1,2}(x_T^*\xi)\subset L^{2}(x_{T}^{*}\xi)$, where $X_{\lambda}$ is the Reeb vector field of $\lambda$ and $\nabla_{t}$ is the covariant derivative along  $x_{T}$. The operator $A_{P}$ does not depend on $\nabla$.

Choosing a $d\lambda$-symplectic trivialization $\Psi$ of $x_{T}^*\xi$ in class $\beta\in\Omega_{x_{T}^{*}\xi}^{+}$, the operator $A_{P}$ is represented as $-J(t)\partial_{t}-S(t)$, for smooth curves $t\in\mathbb{R}/\mathbb{Z}\mapsto J(t)$, $S(t)\in\mathcal{L}_{\mathbb{R}}(\mathbb{C})$, satisfying $J(t)^{2}=-I$ and $\mathrm{det}J(t)=1 \forall t$. If $\Psi$ is unitary then $J(t)\equiv  i$ and $S(t)^{T}=S(t) \forall t$. Therefore, the operator $A_{P}$ takes the form $L_{S}$ as in (\ref{op:auto_adjunto}).  Given  $\eta\in\sigma(A_{P})=\sigma(L_{S})$, a non-vanishing $\eta$-eigensection  is represented in this trivialization by an eigenvector $e\in\mathrm{ker}(L_{S}-\eta I)$ which has a winding number $\mathrm{wind}(\eta, \beta):=\mathrm{wind}(\eta)$.

Defining
\begin{equation*}
\begin{aligned}
\mathrm{wind}^{<0}(A_{P}, \beta):=\mathrm{wind}(\eta^{<0}),\\
\mathrm{wind}^{\geq 0}(A_{P}, \beta):=\mathrm{wind}(\eta^{\geq 0}),
\end{aligned}
\end{equation*}
it follows that
\begin{equation*}
\mu_{CZ}(P, \beta)=\mathrm{wind}^{<0}(A_{P}, \beta)+\mathrm{wind}^{\geq0}(A_{P}, \beta).
\end{equation*}
As before, it does not depend on $\Psi$ in class $\beta$.

\subsection{The rational self-linking number}
Let $(M,\xi=\ker \lambda)$ be a co-oriented contact manifold and let $K\subset M$ be a $k$-unknot transverse to $\xi$ oriented by $\lambda$. Let $u:\D \to M$ be a $k$-disk for $K$ inducing the same orientation and let $Z$ be a smooth non-vanishing section of $u^*\xi$. Given $\epsilon>0$ small, consider the loop $Z_{\epsilon}:\mathbb{R}/\mathbb{Z}\rightarrow M\setminus K$ defined by $Z_{\epsilon}(t):=\mathrm{exp}_{u(e^{2\pi i t})}(\epsilon Z(e^{2\pi it}))$, for some exponential map ${\rm exp}$.
\begin{definition}
The \textit{rational self-linking number} ${\rm sl}(K, u)\in\mathbb{Q}$ is defined as the oriented intersection number
\begin{equation*}
sl(K, u)=\frac{1}{k^{2}} Z_\epsilon\cdot u,
\end{equation*}where $u$  has the orientation induced by $\D$ and $Z_\epsilon$ inherits the orientation of $Z$. The ambient manifold $M$ is oriented by $\lambda \wedge d \lambda$.
\end{definition}
The integer ${\rm sl}(K, u)$ does not depend on $\epsilon$, $\mathrm{exp}$ and $Z$. If $c_{1}(\xi)|_{\pi_2(M)} \equiv 0$ then it does not depend on the $k$-disk $u$ as well. In this case, it is simply denoted by ${\rm sl}(K)$.

\subsection{Pseudo-holomorphic curves in symplectizations}
Let $(M,\xi = \ker \lambda)$ be a co-oriented contact $3$-manifold. Its symplectization is the symplectic manifold $(\mathbb{R}\times M, d(e^{a}\lambda))$, where $a$ is the $\R$-coordinate. Denote by $\pi_{\lambda}:TM\rightarrow\xi$ the projection  along the Reeb direction $X_{\lambda}$, i.e., $\pi(v)=v-\lambda(v)X_\lambda, \forall v\in TM$.

Let $\mathcal{J}_{+}(\xi)$ be the set of $d\lambda$-compatible complex structures on $\xi$. Given $J\in\mathcal{J}_{+}(\xi)$ consider the almost-complex structure $\tilde{J}$ on $\mathbb{R}\times M$ defined by
\begin{equation}\label{estrutura_quase_complexa}
\tilde{J}_{a}\cdot \partial_a = X_{\lambda} \ \ \mbox{ and }\ \ \tilde{J}|_{\xi}\equiv J.
\end{equation}
Let $(S, j)$ be a closed Riemann surface and let $\Gamma\subset S$ be a finite set. A non-constant map $\tilde{u}=(a,u):S\setminus\Gamma\rightarrow\mathbb{R}\times M$ is called a finite energy $\tilde J$-holomorphic curve if it satisfies
\begin{equation*}
\bar{\partial}_{\tilde{J}}(\tilde{u}):=\frac{1}{2}(d\tilde{u}+\tilde{J}(\tilde{u})\circ d\tilde{u}\circ j)=0,
\end{equation*}
and has finite Hofer's energy
\begin{equation*}
E(\tilde{u}):=\sup_{\phi\in\Lambda}\int_{S\setminus\Gamma}\tilde{u}^{*}d(\phi\lambda)<\infty,
\end{equation*}
where $\Lambda:=\{\phi\in C^{\infty}(\mathbb{R}, [0, 1]), \phi'\geq0\}$.

\begin{definition}
The map $\tilde{u}:S\setminus\Gamma\rightarrow\mathbb{R}\times M$ is called somewhere injective if there exists $z_{0}\in S \setminus\Gamma$ so that $d\tilde{u}(z_{0})\neq0$ and $\tilde{u}^{-1}(\tilde{u}(z_{0}))=\{z_{0}\}$.
\end{definition}

The elements in $\Gamma$ are called punctures. A puncture $z\in\Gamma$ is called  positive if $a(\zeta)\to+\infty$ as $\zeta\rightarrow z$. It is called negative if $a(\zeta)\to-\infty$ as $\zeta\rightarrow z$. It is called removable if $a$ is bounded near $z$. If $z$ is removable then $\tilde{u}$ can be smoothly extended over $z$ \cite{hofer1993pseudoholomorphic}. We always assume that no puncture is removable. In this case it is known that every puncture is either positive or negative. If $\tilde u$ is non-constant then Stokes's theorem implies that there always exists a positive puncture.

Let $z\in\Gamma$. Take a holomorphic chart $\varphi:(\D, 0)\rightarrow(U\subset S, z)$ on a neighborhood $U$ of $z$. Then we have positive cylindrical coordinates $(s, t)\in[0, +\infty)\times\mathbb{R}/\mathbb{Z}$ near $z$ given by $(s, t)\simeq\varphi(e^{-2\pi(s+it)})$ and negative cylindrical coordinates $(s, t)\in(-\infty,0]\times\mathbb{R}/\mathbb{Z}$ near $z$ given by $(s, t)\simeq\varphi(e^{2\pi(s+it)})$.

The following result due to Hofer establishes a deep connection between finite energy $\tilde J$-holomorphic curves and closed Reeb orbits.

\begin{theorem}[Hofer \cite{hofer1993pseudoholomorphic}]
Let $(s, t)\in [0,+\infty) \times \R / \Z$ be positive cylindrical coordinates near a puncture $z\in \Gamma$. For each sequence $s_{n}\rightarrow+\infty$, there exists a subsequence, still denoted $s_n$, and a closed Reeb orbit $P=(x,T)$ so that
\begin{equation*}
u(s_n, \cdot )\to x(\epsilon_z T\cdot) \mbox{ in } C^{\infty}_{\rm loc}(\R/ \Z,M) \mbox{ as } n \to +\infty.
\end{equation*}
Here $\epsilon_z=+1$  ($\epsilon_z=-1$) if $z$ is a positive (negative) puncture.
\end{theorem}

\begin{definition}
Let $z\in \Gamma$ be a puncture of a finite energy $\tilde J$-holomorphic curve $\tilde{u}$. Let $\epsilon_{z}\in\{-1, 1\}$ be its sign and let $(s, t)\in [0,+\infty) \times \R / \Z$ be positive cylindrical coordinates near $z$. We say that $z$ is  non-degenerate if the following hold:
\begin{itemize}
\item There exists a closed Reeb orbit $P=(x, T)\in \P(\lambda)$ and $c\in \R$  so that $u(s, t)\to x(\epsilon_{z}Tt)$ and $|a(s,t)-\epsilon_z Ts-c|\to 0$, uniformly in $t$, as $s\to+\infty$.
\item If $\zeta(s, t)\in\xi|_{x(\epsilon_{z}Tt)}$ is defined by $u(s, t)=\mathrm{exp}_{x(\epsilon_{z}Tt)}\zeta(s, t)$ for $s\gg1$ then there exists $b>0$ so that $e^{bs}|\zeta(s, t)|\to 0$, uniformly in $t$, as $s\rightarrow+\infty$.
\item If $\pi\circ du$ does not vanish identically then it does not vanish for $s$ large enough.
\end{itemize}
\end{definition}
The closed Reeb orbit above is called the asymptotic limit of $\tilde u$ at $z$.
If $\lambda$ is non-degenerate then every puncture of $\tilde{u}$ is non-degenerate \cite{properties_1}.

\begin{definition}
Let $P=(x, T)\in\mathcal{P}(\lambda)$ be a closed Reeb orbit. Denote by $T_{\rm min}>0$ the least period of $x$. A Martinet tube for $P$ is a pair $(\mathcal{U}, \Psi)$, where $\mathcal{U}$ is a neighborhood of $x(\mathbb{R})$ in $M$ and $\Psi:\mathcal{U}\rightarrow\mathbb{R}/\mathbb{Z}\times B$ is a diffeomorphism ($B\subset\mathbb{R}^{2}$ is a ball centered at the origin and $\mathbb{R}/\mathbb{Z}\times B$ is provided with coordinates $(\theta, x_{1}, x_{2})$) satisfying
\begin{itemize}
\item $\Psi^{*}(f\cdot (d\theta+x_{1}dx_{2}))=\lambda$, where  $f:\mathbb{R}/\mathbb{Z}\times B\to \mathbb{R}^{+}$ is smooth and satisfies $f|_{\mathbb{R}/\mathbb{Z}\times \{0\}}\equiv T_{\rm min}$ and $df|_{\mathbb{R}/\mathbb{Z}\times \{0\}}\equiv 0$.
\item$\Psi(x(T_{\rm min}t))=(t, 0, 0) \forall t$.
\end{itemize}
\end{definition}

Let $\tilde{u}=(a, u):(S\setminus\Gamma, j)\rightarrow(\mathbb{R}\times M, \tilde{J})$ be a finite energy $\tilde J$-holomorphic curve.  Let $z_0\in \Gamma$ be a puncture of $\tilde u$ with sign $\epsilon\in\{-1,1\}$. Consider positive cylindrical coordinates $(s, t)\in[0,+\infty) \times \R / \Z$ near $z_0$ if $z_{0}$ is positive or consider negative cylindrical coordinates near $z_0$ if $z_{0}$ is negative. Assume that $\lambda$ is nondegenerate and let $P=(x,T)\in \P(\lambda)$ be the asymptotic limit of $\tilde{u}$ at $z_{0}$. Choose a Martinet tube $(\mathcal{U}, \Psi)$ for $P$. For $\epsilon s\gg1$ there are well defined functions $\theta(s, t)\in\mathbb{R}/\mathbb{Z}$, $x_{1}(s, t)$, $x_{2}(s, t)\in\mathbb{R}$,  so that
\begin{equation*}
\Psi\circ u(s, t)=(\theta(s, t), z(s,t))\in \R/\Z \times B \mbox{ with }
z(s, t):=(x_{1}(s, t), x_{2}(s, t)).
\end{equation*}
We still denote by $\theta(s, t)$ its lift to $\mathbb{R}$. Hence it satisfies $\theta(s, t+1)=\theta(s, t)+k$, $\forall (s,t)$, where $k \in \N$ is the covering multiplicity of $P$ determined by $T=kT_{\rm min}$. We may assume that $\theta(s, 0)\to 0$ as $s\to +\infty$. We also assume that $\pi \circ du$ does not vanish near $z_0$.

\begin{theorem}[Hofer, Wysocki and Zehnder \cite{properties_1}, Siefring \cite{siefring2008relative}]\label{theorem:eigen_asymp}
There exist an eigenvalue $\eta$ of $A_{P}$ satisfying $\epsilon\eta<0$, a non-vanishing $\eta$-eigensection $v$, a function $R(s, t)\in\mathbb{R}^{2}$ defined for $\epsilon s\gg1$, and constants $c\in\mathbb{R}$, $r>0$, such that the following hold:
\begin{equation*}
\lim_{\epsilon s\rightarrow+\infty}e^{r|s|}||\partial_{s}^{\beta_{1}}\partial_{t}^{\beta_{2}}[a(s, t)-Ts-c](s, \cdot)||_{L^{\infty}(\mathbb{R}/\mathbb{Z})}=0
\end{equation*}
\begin{equation*}
\lim_{\epsilon s\rightarrow+\infty}e^{r|s|}||\partial_{s}^{\beta_{1}}\partial_{t}^{\beta_{2}}[\theta(s, t)-kt](s, \cdot)||_{L^{\infty}(\mathbb{R}/\mathbb{Z})}=0,
\end{equation*}
 $\forall (\beta_{1}, \beta_{2}) \in \N \times \N$. Moreover, if $t\in\mathbb{R}/\mathbb{Z}\mapsto e(t)\in\mathbb{R}^{2}$ is the representation of $v$ in the frame $\{\partial_{x_{1}}, \partial_{x_{2}}\}$ along $x(T\cdot)$  then
\begin{equation}
z(s, t)=e^{\eta s}(e(t)+R(s, t))\ \mbox{ and } \ \lim_{\epsilon s\rightarrow+\infty}||\partial_{s}^{\beta_{1}}\partial_{t}^{\beta_{2}}R(s, \cdot)||_{L^{\infty}(\mathbb{R}/\mathbb{Z})}=0.
\end{equation}
\end{theorem}

The eigenvalue $\eta$ and the eigensection $v$ are called the asymptotic eigenvalue and the asymptotic eigensection of $\tilde{u}$ at $z_0$, respectively. Theorem \ref{theorem:eigen_asymp} remains valid if we only require $z_0$ to be a nondegenerate puncture.

Let $\beta \in \Omega_{x(T_{\rm min}\cdot)^*\xi}^+$ be a homotopy class of symplectic trivializations of $\xi$ along the prime closed Reeb orbit $P_{\rm min}=(x,T_{\rm min})$. Recall that $k$ is the covering number of $P$ over $P_{\rm min}$. Let $\beta^k$ be the homotopy class of symplectic trivializations of $\xi$ along $P$ induced by $\beta.$ Let ${\rm wind}(v,\beta^k)$ be the winding number of the asymptotic eigensection $v$ of $\tilde u$ at $z_0$ with respect to $\beta^k$.

\begin{definition}[Relatively prime puncture, \cite{hryniewicz2016elliptic}]\label{furos_primos}
The nondegenerate puncture $z_0\in \Gamma$ is called relatively prime if ${\rm gcd}(\wind(v, \beta^k),k)=1$.
\end{definition}

Observe that Definition \ref{furos_primos} does not depend on $\beta$. The following lemma will be very useful in the proof of Theorem \ref{teo:lp}.

\begin{lemma}[Hryniewicz and Salom\~ao \cite{hryniewicz2016elliptic}]\label{mergulho}
Let $z_0\in\Gamma$ be a nondegenerate and relatively prime puncture of a finite energy $\tilde J$-holomorphic curve $\tilde{u}=(a, u):S\setminus\Gamma \rightarrow\mathbb{R}\times M$. Then there exists a  neighborhood $\mathcal{U}$ of $z$ in $S$ so that $u|_{\mathcal{U}\setminus\{z\}}$ is an embedding.
\end{lemma}

We finish this section introducing some algebraic invariants of finite energy curves as in \cite{properties_2}. Let $\tilde{u}=(a, u):(S\setminus\Gamma, j)\to (\mathbb{R}\times M, \tilde{J})$ be a finite energy $\tilde J$-holomorphic curve, where $S$ is closed and connected. Assume that every  puncture of $\tilde u$ is nondegenerate. If $u^{*}d\lambda$ does not vanish identically then $\pi \circ du$ has only  finitely many zeros. Indeed, Carleman's similarity principle implies that the zeros of $\pi\circ du$ are isolated and contribute positively. Moreover, since every $z\in \Gamma$ is nondegenerate,  $\pi\circ du$ does not vanish close to $z$. Hofer, Wysocky and Zehnder define
\begin{equation*}
\mathrm{wind}_\pi(\tilde{u}):=\mbox{algebraic count of zeros of }\pi\circ du.
\end{equation*}
We have
\begin{equation*}
{\rm wind}_\pi(\tilde u) = 0 \Leftrightarrow u \mbox{ is an immersion}.
\end{equation*}

Now take a $d\lambda$-symplectic trivialization $\sigma$ of $u^{*}\xi$. For each $z\in\Gamma$, $\sigma$ induces a homotopy class $\beta_z$ of $d\lambda$-symplectic trivializations of $x_{z}(T_{z}\cdot)^{*}\xi$, where $P_z=(x_{z}, T_{z})$ is the asymptotic limit of $\tilde{u}$ at $z$. Set $\mathrm{wind}(\tilde{u}, z, \sigma):=\mathrm{wind}(v_{z}, \beta_{z})$, where $v_{z}$ is the asymptotic eigensection of $\tilde{u}$ at $z$. Then define
\begin{equation*}
\mathrm{wind}_{\infty}(\tilde{u}):=\sum_{z\in\Gamma^{+}}\mathrm{wind}(\tilde{u}, z, \sigma)-\sum_{z\in\Gamma^{-}}\mathrm{wind}(\tilde{u}, z, \sigma),
\end{equation*}
where $\Gamma^{+}\subset \Gamma$ is the set of positive punctures and $\Gamma^{-}\subset \Gamma$ is the set of negative punctures of $\tilde u$. The above sum does not depend on $\sigma$. In \cite{properties_2} the authors prove that
\begin{equation}\label{wind_eq}
0\leq \mathrm{wind}_{\pi}(\tilde{u})=\mathrm{wind}_{\infty}(\tilde{u})-\chi(S)+\#\Gamma.
\end{equation}

\begin{definition}
A finite energy plane $\tilde{u}:\mathbb{C}\rightarrow\mathbb{R}\times M$ is called fast if the puncture at $\infty$ is  nondegenerate and $\mathrm{wind}_{\pi}(\tilde{u})=0$ (or, equivalently, ${\rm wind}_{\infty}(\tilde u)=1$).
\end{definition}

\section{Proof of main result}\label{dem_lp}

The proof of Theorem \ref{teo:lp} follows the same ideas in \cite{hofer1998dynamics} and \cite{hryniewicz2016elliptic} where the cases $p=1$ and $p=2$ are treated, respectively. Here we assume $p>2$.

Let $\pi_p:S^3 \to L(p,1)$ be the natural projection. Let $f:L(p, 1)\to (0, +\infty)$ be the smooth function so that $\lambda =f \lambda_0$, where $\lambda_0$ is the Liouville form on $L(p,1)$. 

Let $H:\mathbb{R}^{4}\rightarrow\mathbb{R}$ be the Hamiltonian function
\begin{equation*}
H(x_{1}, x_{2}, y_{1}, y_{2})= \frac{x_{1}^{2}+y_{1}^{2}}{r_{1}^{2}}+\frac{x_{2}^{2}+y_{2}^{2}}{r_{2}^{2}},
\end{equation*}
where $0<r_{1}<r_{2}\in\mathbb{R}$. Let $E:=H^{-1}(1)\in\mathbb{R}^{4}$ be the ellipsoid associated to $r_1$ and $r_2$. It induces a contact form $\lambda_{E}=f_{E}\lambda_{0}$ on $S^{3}$ by pulling back the Liouville form via the radial map $S^3 \to E$. Since $\lambda_{E}$ is preserved under $g_{p,1}$, it descends to a contact form on $L(p,1)$, still denoted  $\lambda_{E}$.

If $r_{2}^2/r_{1}^{2}\in\mathbb{R} \setminus\mathbb{Q}$ then the Reeb flow of $\lambda_{E}$ on $L(p,1)$ has precisely two nondegenerate simple closed Reeb orbits $P_{1}=\pi_{p, 1}(S^{1}\times\{0\})$ and $P_{2}=\pi_{p, 1}(\{0\}\times  S^{1})$, forming a Hopf link. Their periods are $T_{1}=\pi r_{1}^{2}/p$ and $T_{2}=\pi r_{2}^{2}/p$, respectively. The closed Reeb orbits $P_{1}$ and $P_{2}$ are $p$-unknotted and have self-linking number $-1/p$ \cite{HLS}. Moreover, the Conley-Zehnder indices of the contractible closed Reeb orbits orbits $P_{1}^{p}$ and $P_{2}^{p}$ are
\begin{equation*}
\mu_{CZ}(P_{1}^{p})=3 \ \mbox{ and } \mu_{CZ}(P_{2}^{p})=2k+1,
\end{equation*}
where $k\in \mathbb{N}$ is such that $r_{2}^{2}/r_{1}^{2}\in(k-1, k)$. It is always possible to choose $r_{1},r_{2}$ large enough so that $f<f_{E}$ on $L(p, 1)$, see \cite{coloquio_contato}.

\subsection{The non-degenerate case}

In this section we prove Theorem \ref{teo:lp} in the non-degenerate case. The degenerate case will be treated as a limiting case.

\begin{proposition}\label{prop:nao_deg}
Let $\lambda=f\lambda_{0}$ be a contact form on $L(p,1)$, where $f:L(p, 1)\rightarrow (0, +\infty)$ is smooth. Choose $0<r_{1}<r_{2}$ and $r_{2}^{2}/r_{1}^{2}\in\mathbb{R}\setminus\mathbb{Q}$ so that $f<f_{E}$ pointwise. Assume that every closed Reeb orbit $P\in\mathcal{P}(\lambda)$ with period $\leq \pi r_{1}^{2}$ satisfies:
\begin{itemize}
\item[(i)] $P$ is nondegenerate.
\item[(ii)] If $P$ is contractible then $\mu_{CZ}(P, \beta_{\rm disk})\geq3$.
\end{itemize}
Let $J\in \mathcal{J}_+(\xi)$  and let $\tilde{J}$ be the cylindrical almost complex structure on $\mathbb{R}\times L(p, 1)$ induced by $\lambda$ and $J$. Then there exists a $p$-unknotted closed Reeb orbit $P_0=(x_0,T_0)\in \P(\lambda)$ and a finite energy $\tilde{J}$-holomorphic plane $\tilde{u}:\mathbb{C} \to \mathbb{R}\times L(p, 1)$ satisfying
\begin{itemize}

\item The asymptotic limit of $\tilde{u}$ at $\infty$ is the nondegenerate closed Reeb orbit $P_{0}^{p}=(x_{0}, pT_{0})$ whose Conley-Zehnder index is $3$.
\item $\tilde{u}$ is embedded and $E(\tilde{u})=pT_{0}\leq\pi r_{1}^{2}$.
\item $u$ is an embedding transverse to $X_{\lambda}$ and $u(\mathbb{C})\cap x_{0}(\mathbb{R})=\emptyset$.
\item the rational self-linking number of $P_0$ is $-1/p$.
\end{itemize}
\end{proposition}

The first step in the proof of Proposition \ref{prop:nao_deg} consists of constructing a symplectic cobordism between  $\lambda$ and $\lambda_{E}$, see \cite{hofer1998dynamics, hryniewicz2016elliptic} for more details. Take a smooth function $h:\mathbb{R}\times L(p, 1)\rightarrow\mathbb{R}^{+}$ satisfying
\begin{itemize}
\item $h(a, \cdot)=f$, if $a\leq-2$.
\item $h(a, \cdot)=f_{E}$, if $a\geq2$.
\item $\frac{\partial h}{\partial a}\geq0$.
\item $\frac{\partial h}{\partial a}>\sigma>0$ on $[-1, 1]\times L(p, 1)$ for some $\sigma>0$.
\end{itemize}

The family of contact forms $\lambda_{a}:=h(a, \cdot)\lambda_{0}$, $a\in\mathbb{R}$, interpolates $\lambda$ and $\lambda_{E}$ and the contact structure $\xi=\ker\lambda_{a}$ does not depend on $a$. Choose $J_{E}\in\mathcal{J}_+(\lambda_{E})$ and consider a family $J_{a}\in\mathcal{J}_+(\lambda_{a})$, $a\in\mathbb{R}$,  such that $J_{a}=J$ if $a\leq-2$ and $J_{a}=J_{E}$ if $a\geq2$. Consider a smooth almost complex structure $\bar{J}$ on  $\mathbb{R}\times L(p, 1)$ with the following properties: \begin{itemize}
\item $\bar J = \tilde J_a$ on $(\mathbb{R}\setminus[-1, 1]) \times L(p,1)$, where $\tilde J_a$ satisfies $\tilde J_a \cdot \partial_a = X_{\lambda_a}$ and $\tilde J_a|_\xi = J_a$.
\item $\bar J$ is compatible with the symplectic form $d(h\lambda_{0})$ on $[-1, 1]\times L(p, 1)$.
\end{itemize}
The space of such almost complex structures $\bar J$ is non-empty and contractible in the $C^{\infty}$-topology and is denoted by $\mathcal{J}(\lambda_{E}, J_{E}, \lambda, J)$.

Let $\Gamma\in\mathbb{C}$ be a finite set. Consider non-constant maps
\begin{equation*}
\tilde{u}:\mathbb{C}\setminus\Gamma\rightarrow\mathbb{R}\times L(p, 1),
\end{equation*}
which satisfy
\begin{equation}\label{eq:pseudo_holomorfa}
d\tilde{u}\circ i=\bar{J}(\tilde{u})\circ d\tilde{u},
\end{equation}
for some  $\bar{J}\in\mathcal{J}(\lambda_{E}, J_{E}, \lambda, J)$, with finite Hofer's energy $0<E(\tilde{u})<\infty$. The energy $E(\tilde{u})$ is defined as follows: let $\Lambda$ be the space of smooth functions $\phi:\mathbb{R} \rightarrow [0, 1]$ satisfying $\phi '\geq0$ and $\phi=1/2$ on $[-1, 1]$. Then
\begin{equation*}
E(\tilde{u}):=\sup_{\phi\in\Lambda}\int_{\mathbb{C}\setminus \Gamma}\tilde{u}^{*}d(\phi(a)\lambda_{a}),
\end{equation*}
where $\lambda_{a}$ is seen as $1$-form on $\mathbb{R}\times L(p, 1)$. These maps are called generalized finite energy spheres and if $\Gamma=\emptyset$ then they are called generalized finite energy planes.

As in the case of cylindrical almost complex structures, the set of non-removable punctures $\Gamma$ of a generalized finite energy sphere $\tilde{u}=(a, u)$ is non-empty and if $z_{0}\in \Gamma$ then either $a(z)\rightarrow +\infty$ or $a(z)\rightarrow -\infty$ as $z\rightarrow z_{0}\in\Gamma$. Therefore, it makes sense to talk about positive and negative punctures of $\tilde u$. Moreover, if $\tilde{u}$ is a generalized finite energy sphere then $\tilde{u}$ has at least one positive puncture due to the exact nature of the symplectic cobordism.

\begin{proposition}[Hofer-Wysocki-Zehnder \cite{hofer1998dynamics}]
If $\tilde{u}=(a, u):\mathbb{C}\rightarrow\mathbb{R}\times L(p, 1)$ is a generalized finite energy plane then:
 \begin{itemize}
 \item $a(z)\rightarrow+\infty$ as $|z|\rightarrow+\infty$.
  \item $u(Re^{2\pi i\cdot})\rightarrow x(T\cdot)$ in $C^{\infty}(\R / \Z)$ as $R\rightarrow+\infty$.
  \end{itemize}
  Here $P=(x,T)$ is a contractible closed Reeb orbit of $\lambda_{E}$ whose period satisfies $T= E(\tilde{u})$.
\end{proposition}

Fix $\bar{J}\in \mathcal{J}(\lambda,J,\lambda_E,J_E)$. Denote by $\Theta$ the space of $\bar{J}$-holomorphic generalized finite energy planes asymptotic to $P_{1}^{p}$, modulo holomorphic re-parameterizations. Recall that $P_{1}\subset L(p, 1)$ is an elliptic non-contractible closed Reeb orbit of $\lambda_E$ with period $\pi r_{1}^{2}/p$. Moreover, the Conley-Zehnder index of its $p$-th iterate is $\mu_{CZ}(P_{1}^{p})=3$. The energy of each $\tilde u \in \Theta$ is $E(\tilde{u})=\pi r_{1}^{2}$.

Following \cite{hofer1998dynamics} and \cite{hryniewicz2016elliptic},  $\Theta$ has the structure of a $2$-dimensional manifold. Moreover,  using the intersection theory of holomorphic curves developed by Siefring in \cite{siefring2008relative, siefring2011intersection}, we can fix the connected component $\Theta' \subset\Theta$ containing the finite-energy planes in the ellipsoidal positive end. The planes in this component are embedded and do not intersect each other.

\subsubsection{Bubbling-off tree} At this moment we are interested in the compactness properties of the component $\Theta'\subset \Theta$. In order to do that we introduce a bubbling-off tree of pseudo-holomorphic curves. Consider an oriented, rooted and finite tree $\mathcal{T}$, and a finite set $\mathcal{U}$ of finite energy $\tilde J$-holomorphic spheres so that the following properties are satisfied:
\begin{itemize}
\item There exists a bijective correspondence between  vertices $q\in\mathcal{T}$ and finite energy spheres $\tilde{u}_{q}\in\mathcal{U}$. Every sphere $\tilde{u}_{q}:\mathbb{C}\setminus\Gamma_{q} \rightarrow\mathbb{R}\times L(p, 1)$ is pseudo-holomorphic with respect to $\tilde{J}_{E}$,  $\bar{J}$ or $\tilde{J}$. Moreover, every oriented path $(q_{1}, ..., q_{N})$ from the root $q_{1}=r$ to a leaf $q_{N}\in \mathcal{T}$ has at most one vertex $q_{i}$ for which $\tilde{u}_{q_{i}}$ is $\bar{J}$-holomorphic. In that case, $\tilde{u}_{q_{j}}$ is $\tilde{J}_{E}$-holomorphic $\forall1\leq j<i$ and $\tilde{u}_{q_{j}}$ is $\tilde{J}$-holomorphic $\forall i<j\leq N$.
\item Every sphere $\tilde{u}_{q}$ has precisely one positive puncture at $\infty$ and $0\leq \#\Gamma_{q}<+\infty$ negative punctures.
\item If the vertex $q$ is not a root, then $q$ has an income edge $e$ that comes from a vertex $q'$ and $\#\Gamma_{q}$ outgoing edges $f_{1}, ..., f_{\#\Gamma_{q}}$ that go  to vertices $p_{1}, ..., p_{\#\Gamma_{q}}\in\mathcal{T}$, respectively. The edge $e$ is associated to the positive puncture of $\tilde{u}_{q}$ and the edges $f_{1}, ..., f_{\#\Gamma_{q}}$ are associated to the $\#\Gamma_q$ negative punctures of $\tilde{u}_{q}$. The asymptotic limit of $\tilde{u}_{q}$ at the positive puncture coincides with the asymptotic limit of $\tilde{u}_{q'}$ at its negative puncture associated to $e$. In the same way, the asymptotic limit of $\tilde{u}_{q}$ at its negative puncture corresponding to $f_{i}$ coincides with the asymptotic limit of $\tilde{u}_{p_{i}}$ at its unique positive puncture. If $\tilde{u}_{q}$ is $\tilde{J}_{E}$-holomorphic then $\tilde{u}_{p_{i}}$ is either $\tilde{J}_{E}$ or $\bar{J}$-holomorphic. If $\tilde{u}_{q}$ is $\bar{J}$ or $\tilde{J}$-holomorphic then $\tilde{u}_{p_{i}}$ is necessarily $\tilde{J}$-holomorphic.
\item If $\tilde{u}_{q}$ is $\tilde{J}_{E}$ or $\tilde{J}$-holomorphic and its contact area vanishes then $\#\Gamma_{q}\geq 2$.
\end{itemize}

Denote by $\mathcal{B}=(\mathcal{T}, \mathcal{U})$ the \textit{bubbling-off} tree, where $\mathcal{T}$ and $\mathcal{U}$ satisfies the above properties. Take a sequence $\tilde{u}_{n}=(a_{n}, u_{n})$ of generalized finite energy planes representing elements of $\Theta'$. It follows that $E(\tilde{u}_{n})=\pi r_{1}^{2}, \forall n.$ Hence the SFT compactness theorem in \cite{bourgeois2003compactness} implies the following theorem.

\begin{theorem}\label{teo_bub_tree}
There exists a bubbling-off tree $\mathcal{B}=(\mathcal{T}, \mathcal{U})$ so that, up to a subsequence of $\tilde{u}_{n}$, the following hold:
\begin{itemize}
\item Given a vertex $q\in\mathcal{T}$ there exist sequences $z_{n}^{q}$, $\delta_{n}^{q}\in\mathbb{C}$ and $c_{n}^{q}\in\mathbb{R}$ such that
\begin{equation}\label{eq:sequencia}
\tilde{u}_{n}(z_{n}^{q}+\delta_{n}^{q}\cdot)+c_{n}^{q} \rightarrow \tilde{u}_{q} \mbox{ in } C^\infty_{\rm loc}(\C \setminus \Gamma_q) \mbox{ as } n \to +\infty.
\end{equation}
Here, $\tilde{u}+c:=(a+c, u)$.
\item The curve $\tilde{u}_{r}$ associated to the root $r\in \mathcal{T}$ is asymptotic to $P_{1}^{p}$ at $\infty$. Moreover, every asymptotic limit of every $\tilde{u}_{q}$ is a contractible closed Reeb orbit of $\lambda_E$ or $\lambda$  and its period is $\leq\pi r_{1}^{2}$.
\end{itemize}
\end{theorem}
The bubbling-off tree $\mathcal{B}=(\mathcal{T}, \mathcal{U})$ given in Theorem \ref{teo_bub_tree} is called a SFT-limit of $\tilde u_n$.  Due to a deep analysis found in \cite{hofer1998dynamics} and \cite{hryniewicz2016elliptic} in the cases $p=1$ and $p=2$, respectively, it is possible to obtain a much better description of the bubbling-off tree $\mathcal{B}$. Given sequences $z_{n}^{q}$, $\delta_{n}^{q}\in \C$ and $c_{n}^{q}\in\R$ as in \eqref{eq:sequencia}, it is verified that precisely one of the following alternatives holds:
\begin{enumerate}
\item[(I)] $c_{n}^{q}$ is bounded, $a_{n}(z_{n}^{q}+\delta_{n}^{q}\cdot)$ is $C_{\rm loc}^{0}(\mathbb{C}\setminus \Gamma_{q})$-bounded in $n$ and $\tilde{u}_{q}$ is $\bar{J}$-holomorphic.
\item[(II)] $c_{n}^{q}\to-\infty$, $a_{n}(z_{n}^{q}+\delta_{n}^{q}\cdot)\to +\infty$ in $C_{\rm loc}^{0}(\mathbb{C}\setminus \Gamma_{q})$ as $n\to +\infty$, and $\tilde{u}_{q}$ is $\tilde{J}_{E}$-holomorphic.
\item[(III)] $c_{n}^{q}\to +\infty$, $a_{n}(z_{n}^{q}+\delta_{n}^{q}\cdot)\to-\infty$ in $C_{\rm loc}^{0}(\mathbb{C}\setminus \Gamma_{q})$ as $n\to +\infty$, and $\tilde{u}_{q}$ is $\tilde{J}$-holomorphic.
\end{enumerate}
If $q\in \mathcal{T}$ is a vertex for which (III) is verified then $\tilde{u}_{q}$ is asymptotic at its positive puncture to a closed Reeb orbit of $\lambda$ whose period is strictly less than $\pi r_{1}^{2}$. Hence (III) cannot be verified for the root $r$. Furthermore, as in the cases $p=1$ and $p=2$, the following important result holds for the connected component $\Theta'\subset \Theta$.

\begin{proposition}[\cite{hofer1998dynamics}, \cite{hryniewicz2016elliptic}]\label{prop_mininf}
There exists a sequence $\tilde{u}_{n}=(a_{n}, u_{n})\in\Theta'$ so that
\begin{equation}\label{seq_bubbling}
\min_{z\in\mathbb{C}}a_{n}(z)\rightarrow-\infty \mbox{ as }n\to+\infty.
\end{equation}
\end{proposition}

The proof of Proposition \ref{prop_mininf} essentially follows from the SFT compactness theorem and the Fredholm theory developed for embedded holomorphic curves in symplectizations \cite{properties_3}.

Results from Hofer, Wysocki e Zehnder in \cite{properties_3} guarantee the following.

\begin{theorem}\label{teo:fred}
There exists a dense subset $\mathcal{J}_{\rm reg}\subset \mathcal{J}(\lambda_{E}, J_{E}, \lambda, J)$ so that if $\bar{J}\in\mathcal{J}_{reg}$ and $\tilde{u}:\mathbb{C}\setminus\Gamma\rightarrow \mathbb{R}\times L(p, 1)$ is a somewhere injective generalized finite energy $\bar J$-holomorphic sphere  then
\begin{equation}\label{fred_ind}
\mathrm{Fred}(\tilde{u}):=\mu_{CZ}(P_{\infty})-\sum_{z\in\Gamma} \mu_{CZ}(P_{z})+\#\Gamma-1\geq0.
\end{equation}
The Conley-Zehnder indices above are computed using a trivialization which extends over $u^*\xi$.
\end{theorem}

We assume from now on that  $\bar{J}\in\mathcal{J}_{reg}$ and so \eqref{fred_ind} holds.

\begin{theorem}\label{teo:furo_negativo}
The bubbling-off tree $\mathcal{B}=(\mathcal{T}, \mathcal{U})$ obtained as a SFT limit of a sequence $\tilde u_n=(a_n,u_n)\in \Theta'$ satisfying  \eqref{seq_bubbling} contains only two vertices $r$ and $q$. The root $r$ corresponds to a finite energy $\bar{J}$-holomorphic cylinder $\tilde{u}_{r}:\C\setminus\{0\} \to\R\times L(p, 1)$ which is asymptotic to $P_{1}^{p}$ at its positive puncture $\infty$ and to a closed orbit $P\in\mathcal{P}(\lambda)$ at its negative puncture $0$. The Conley-Zehnder index of $P$ is $3$.  The vertex $q$ corresponds to an embedded finite energy $\tilde{J}$-holomorphic plane $\tilde{u}_{q}=(a_q,u_q):\C\to \R\times L(p, 1)$ asymptotic to $P$ at $\infty$. There exists a $p$-unknotted closed Reeb orbit $P'\in \P(\lambda)$ so that $P=(P')^p$ and $u_q:\C \to L(p,1)$ is a $p$-disk for $P'$.
\end{theorem}

\begin{proof}
Let $\tilde{u}_{r}=(a_{r}, u_{r}):\mathbb{C}\setminus \Gamma_{r} \rightarrow\mathbb{R}\times L(p, 1)$ be the finite energy sphere associated to the root $r\in \mathcal{T}$. Then $P_{1}^{p}$ is the asymptotic limit of $\tilde{u}_{r}$ at its positive puncture. Next we show that $\tilde{u}_{r}$ is $\bar{J}$-holomorphic.

There exist sequences $z_{n}^{r}$, $\delta_{n}^{r}\in \C$ and $c_{n}^{r}\in \R$ so that (\ref{eq:sequencia}) holds for $\tilde u_r$. We know that either (I) or (II) holds. Suppose that (II) holds. Then $\tilde{u}_{r}$ is $\tilde{J}_{E}$-holomorphic. If the $d\lambda_E$-area of $\tilde u_r$ vanishes then $\#\Gamma_r \geq 2$. Since every asymptotic limit at a puncture in $\Gamma_r$ is contractible, its action is $\geq \pi r_1^2$. This implies $\int_{\C \setminus \Gamma_r} u_r^* d\lambda_E \leq \pi r_1^2 - \# \Gamma_r \pi r_1^2 <0$, a contradiction. If the $d\lambda_E$-area of $\tilde u_r$ does not vanish then we claim that $\Gamma_r =\emptyset$. If this is not the case then the action of the  asymptotic limit $P_-$ at a negative puncture is strictly less than $\pi r_1^2$. Since $P_-$ must be contractible, we get a contradiction. Indeed, $\pi r_1^2$ is the smallest action of a contractible closed Reeb orbit of $\lambda_E$.

Now suppose that $\tilde{u}_{r}$ is a finite energy $\tilde{J}_{E}$-holomorphic plane. In this case, due to Lemma 3.13 in \cite{hryniewicz2016elliptic} and the absence of bubbling-off points for the sequence $\tilde u_n$, given $M>0$ large there exists $R>0$ so that $a_{n}(z_{n}^{r}+\delta_{n}^{r}z)+c_{n}^{r}>M$ if $|z|>R_{0}$ and $a_{n}(z_{n}^{r}+\delta_{n}^{r}z)>0$ if $|z|\leq R_{0}$, for all large $n$. Since $c_{n}^{r}\to -\infty$, we obtain $\inf_n \min_{z\in \C} a_{n}(z)\geq 0$, contradicting \eqref{seq_bubbling}. We conclude that (I) holds for $\tilde u_r$.

Now we check that $\tilde{u}_{r}$ has at least one negative puncture. In fact, if $\Gamma_{r}=\emptyset$ then $\tilde{u}_{r}$ is a finite energy $\bar{J}$-holomorphic plane. Since $c_{n}^{r}$ is bounded, it follows from Lemma 3.13 in \cite{hryniewicz2016elliptic} that there exists $R_{0}>0$ so that $a_{n}(z_{n}^{r}+ \delta_{n}^{r}z)\in [0, +\infty)$ if $|z|>R_{0}$, for all large $n$. However, $a_{n}(z_{n}^{r}+ \delta_{n}^{r}\cdot)$ is uniformly bounded on compact subsets of $\mathbb{C}$, in particular, on $\{|z|\leq R_{0}\}$. This contradicts \eqref{seq_bubbling}. Hence  $\Gamma_{r}\neq\emptyset$.

So far we know that $\tilde u_r$ is $\bar J$-holomorphic and has at least one negative puncture. In the following we consider two different cases. Suppose first that $\tilde{u}_{r}$ is not somewhere injective, the other case will be treated later. Then we find a somewhere injective $\bar J$-holomorphic curve $\tilde{v}_{r}:\C\setminus\Gamma'\to \R\times L(p, 1)$, $1\leq\#\Gamma'$, and a polynomial $Q:\C \to \C$ with degree ${\rm deg} (Q) \geq 2$ satisfying  $\tilde{u}_{r}=\tilde{v}_{r}\circ Q$ and $Q^{-1}(\Gamma')=\Gamma_{r}$. All punctures in $\Gamma'$ are negative and the corresponding asymptotic limits are closed Reeb orbits of $\lambda$. Furthermore, the asymptotic limit of $\tilde v_r$ at $\infty$ is $P_1^{k}$, where $k = p/ {\rm deg}(Q)$.

Denote by $\Gamma'=\{w_1,\ldots,w_n\}\subset \C$ the set of negative punctures of $\tilde v_r$ and let $P_{w_i}=(x_{w_i}, T_{w_i})$ be the asymptotic limit of $\tilde v_r$ at $w_i$, $i=1,\ldots, n$. Denote by $P_\infty=(x_\infty,T_\infty)=P_1^k=(x_1,kT_{\rm min})$ the asymptotic limit of $\tilde v_r$ at $\infty$. We shall need the following lemma whose proof is found in Appendix \ref{appendix_a}.

\begin{lemma}\label{lema:unico_furo}
The curve $\tilde v_r$ contains precisely one negative puncture whose asymptotic limit is non-contractible. Its positive asymptotic limit $P_\infty=P_1^k$ is also non-contractible.
\end{lemma}

Assume that $w_1$ is the puncture given in Lemma \ref{lema:unico_furo}. In particular, $P_{w_1}$ is non-contractible and $P_{w_2},\ldots P_{w_n}$ are contractible, see Figure \ref{fig:esfera}. For each $i=2,\ldots,n$, we choose a symplectic trivialization of $x_{w_i}(T_{w_i}\cdot)^{*}\xi$ which extends over a capping disk $f_i:\D \to L(p,1)$ for $P_{w_i}$. Since $\tilde{v}_{r}\# f_2 \# \ldots \#f_n$ is topologically a cylinder,  each class $\beta$ of symplectic trivializations of $x_{\infty}(T_{\infty}\cdot)^{*}\xi$ induces a class $\beta'$ of symplectic trivializations of $x_{w_1}(T_{w_1}\cdot)^{*}\xi$.

\begin{figure}
    \includegraphics[width=0.45\textwidth]{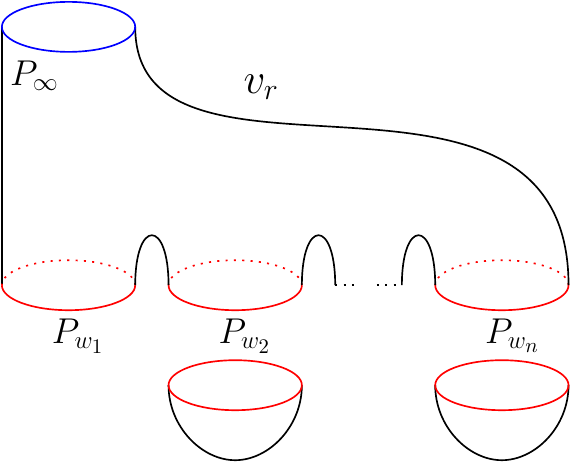}
    \caption{The $\bar J$-holomorphic curve $\tilde{v}_{r}$. $P_{w_1}$ is non-contractible and $P_{w_2},\ldots, P_{w_n}$ are contractible.}
    \label{fig:esfera}
\end{figure}

\begin{lemma}\label{lema:trivializacao}
There exists a class $\beta$ of $d\lambda_E$-symplectic trivializations of $x_{1}(T_{\rm min}\cdot)^{*}\xi$  such that  $\mu_{CZ}(P_{1}^{j}, \beta^j)=2j-1$ for every $j=1.\ldots p$. Moreover, $\mathrm{wind}(\beta^p, \beta_{\rm disk})=2-p$.
\end{lemma}

The proof of Lemma \ref{lema:trivializacao} is found in Appendix \ref{appendix_b}. Due to Lemma \ref{lema:trivializacao} we find a class $\beta^{k}$ of symplectic trivializations of $x_{\infty}(T_{\infty}\cdot)^{*}\xi$ so that
\begin{equation*}
\mu_{CZ}(P_\infty, \beta^k)=\mu_{CZ}(P_{1}^{k}, \beta^k)=2k-1.
\end{equation*}
The class $\beta^k$ induces a class $\beta'$ of symplectic trivializations of $x_{w_1}(T_{w_1}\cdot)^{*}\xi$.

\begin{proposition}\label{lema:ind_2k_1}
$\mu_{CZ}(P_{w_{1}}, \beta')=2k-1$ and $n=1$.
\end{proposition}
\begin{proof}
Using the class $\beta$ as in Lemma \ref{lema:trivializacao} and the dynamical convexity of $\lambda$ we obtain
\begin{equation*}
\begin{aligned}
0 & \leq  \mu_{CZ}(P_\infty, \beta^k)-\mu_{CZ}(P_{w_{1}}, \beta')-\sum_{i=2}^n \mu_{CZ}(P_{w_i},\beta_{\rm disk})+\#\Gamma'-1\\
&= \mu_{CZ}(P_1^{k}, \beta^k)-\mu_{CZ}(P_{w_{1}}, \beta')-\sum_{i=2}^n\mu_{CZ}(P_{w_i},\beta_{\rm disk})+\#\Gamma'-1 \\
&\leq  2k-1-\mu_{CZ}(P_{w_{1}}, \beta')-3(n-1)+n-1 \\
&=-\mu_{CZ}(P_{w_{1}}, \beta')-2(n-1)+2k-1,
\end{aligned}
\end{equation*}
from where we conclude that
\begin{equation}\label{ineq:indice_2k_1}
\mu_{CZ}(P_{w_{1}}, \beta')\leq 2k-1 -2(n-1) \leq 2k-1.
\end{equation}

Now the class $\beta^{k}$ induces the class $(\beta^{k})^{p/k}=\beta^{p}$ of symplectic trivializations of $x_{\infty}(\frac{p}{k}(kT_{\rm min})\cdot)^{*}\xi=x_{\infty}(pT_{\rm min}\cdot)^{*}\xi$ such that $\mu_{CZ}(P_{\infty}^{p/k}, \beta^p)=\mu_{CZ}(P_1^p,\beta^p)=2p-1$, see Lemma \ref{lema:trivializacao}. Furthermore, $\mathrm{wind}((\beta^k)^{p/k}, \beta_{\rm disk})=2-p$. By the hypothesis on $\lambda$, we know that $\mu_{CZ}(P_{w_{1}}^{p/k}, \beta_{\rm disk})\geq3$. Therefore,  it follows from \eqref{relacao_conley_bases} that
\begin{equation}\label{ineq:ind_minimo}
\begin{aligned}
\mu_{CZ}(P_{w_{1}}^{p/k}, (\beta')^{p/k}) & = \mu_{CZ}(P_{w_{1}}^{p/k}, \beta_{\rm disk})+2\wind(\beta_{\rm disk},(\beta')^{p/k})\\
& \geq 3-2\mathrm{wind}((\beta')^{p/k}, \beta_{\rm disk})\\
& = 3-2(2-p)=2p-1.
\end{aligned}
\end{equation}

The periodic orbit $P_{w_{1}}$ is non-degenerate and hence it is either hyperbolic or elliptic. Assume that $P_{w_{1}}$ is hyperbolic and suppose by contradiction that
\begin{equation*}
\mu_{CZ}(P_{w_{1}}, \beta')<2k-1.
\end{equation*}
From (\ref{relacao_conley_rotacao}) and the inequality above, it follows that
\begin{equation*}
\mu_{CZ}(P_{w_{1}}^{p/k}, (\beta')^{p/k})=\frac{p}{k}\mu_{CZ}(P_{w_{1}}, \beta')<p\frac{(2k-1)}{k}\leq 2p-1,
\end{equation*}
which is in contradiction with \eqref{ineq:ind_minimo}.

Assume now that $P_{w_{1}}$ is elliptic. In view of (\ref{ineq:ind_minimo})  we have
\begin{equation*}
\rho(P_{w_1}^{p/k},(\beta')^{p/k}) \geq p-1
\end{equation*}
and hence
\begin{equation}\label{rotacao_k-1}
\rho(P_{w_{1}}, \beta')\geq k\frac{p-1}{p} \geq k-1.
\end{equation}
Therefore, it follows from (\ref{relacao_conley_rotacao}) and (\ref{rotacao_k-1}) that
\begin{equation*}
\mu_{CZ}(P_{w_{1}}, \beta')=2\lfloor\rho(P_{w_{1}}, \beta')\rfloor+1\geq2k-1.
\end{equation*}
Using \eqref{ineq:indice_2k_1} we conclude that $\mu_{CZ}(P_{w_{1}}, \beta')=2k-1$ and $n=1$.
\end{proof}

Proposition \ref{lema:ind_2k_1} implies  $\Gamma'=\{w_1\}$. Furthermore, we know that $\mu_{CZ}(P_{w_1},\beta')=2k-1$. If $P_{w_1}$ is hyperbolic then \eqref{relacao_conley_rotacao} implies that $\rho(P_{w_1},\beta')=\mu_{CZ}(P_{w_1},\beta')/2=k-1/2$. It follows that
\begin{equation*}
\begin{aligned}
\mu_{CZ}(P_{w_{1}}^{p/k}, (\beta')^{p/k}) & = 2\rho(P_{w_1}^{p/k},(\beta')^{p/k})=2\frac{p}{k}\rho(P_{w_1},\beta')\\
& =2\frac{p}{k}\left(k-\frac{1}{2}\right)=2p-\frac{p}{k}\leq 2p-1.
\end{aligned}
\end{equation*}
Now if $P_{w_1}$ is elliptic then \eqref{relacao_conley_rotacao} implies  $\lfloor \rho(P_{w_1},\beta') \rfloor=(\mu_{CZ}(P_{w_1},\beta')-1)/2=k-1$ and hence $\rho(P_{w_1},\beta')<k$. Hence
\begin{equation*}
\begin{aligned}
\mu_{CZ}(P_{w_{1}}^{p/k}, (\beta')^{p/k}) & = 2\lfloor \rho((P_{w_1})^{p/k},(\beta')^{p/k})\rfloor +1=2\lfloor \frac{p}{k}\rho(P_{w_1},\beta')\rfloor +1\\
& \leq 2(p-1) +1 = 2p-1.
\end{aligned}
\end{equation*}
Using (\ref{ineq:ind_minimo}) we  conclude that $\mu_{CZ}(P_{w_{1}}^{p/k}, (\beta')^{p/k})=2p-1$. Lemma \ref{lema:trivializacao} and \eqref{relacao_conley_bases} imply the following crucial fact
\begin{equation}\label{mu=3}
\mu_{CZ}(P_{w_{1}}^{p/k}, \beta_{\rm disk}) =2p-1+2(2-p) = 3.
\end{equation}

Recall that ${\rm deg}(Q)=p/k$. We claim that the polynomial $Q$ has the form
\begin{equation}\label{pol_Q}
Q(z)=w_1+ a(z-z_{1})^{p/k},
\end{equation}
for some $a\in \C^*$ and $z_1\in \C$. If this is not the case, then we write $Q(z)=w_1+b\prod_{i=1}^l(z-z_i)^{m_i}$ for some $b\in \C^*, m_i\in \N^*$, $z_i\neq z_j\in \C, \forall i\neq j,$ and $l\geq 2$. We have $\sum_{i=1}^l m_i = p/k$, therefore $m_{i}<p/k$, for $i=1, \cdots, l$. Observe that $\Gamma_r = \{z_1,\ldots,z_l\}$ is the set of punctures of $\tilde u_r$. Since $P_{z_{i}}$ is contractible, we have
\begin{equation*}
3\leq\mu_{CZ}(P_{z_{i}}, \beta_{\rm{disk}})=\mu_{CZ}(P_{w_{1}}^{m_{i}}, \beta_{\rm{disk}}),
\end{equation*}
but it implies that $5\leq\mu_{CZ}(P_{w_{1}}^{p/k},\beta_{\rm{disk}})$, which contradicts \eqref{mu=3}. Therefore we conclude that $\Gamma_r = \{z_1\}$ and $Q$ has the form \eqref{pol_Q}. Moreover, the asymptotic limit of $\tilde u_r$ at $z_1$ is the contractible closed Reeb orbit $P:=P_{w_1}^{p/k}$ which, in view of \eqref{mu=3}, has Conley-Zehnder index $3$.

Now we deal with the case where $\tilde{u}_{r}$ is somewhere injective. The asymptotic limit $P_{z}$ of $\tilde{u}_{r}$ at a negative puncture $z\in\Gamma_{r}$ is a contractible periodic orbit, its period is $\leq \pi r_1^2$ and, therefore, $\mu_{CZ}(P_{z},\beta_{\rm disk})\geq 3$. Using Theorem \ref{teo:fred} and that the positive asymptotic limit is $P_1^p$, with $\mu_{CZ}(P_1^p, \beta_{\rm disk})=3$, we obtain
\begin{equation*}
\begin{aligned}
0 & \leq\mu_{CZ}(P_1^p, \beta_{\rm disk})-\sum_{z\in\Gamma_r} \mu_{CZ}(P_{z}, \beta_{\rm disk})+ \#\Gamma_{r}-1\\
& =2-\sum_{z\in\Gamma_r} \mu_{CZ}(P_{z}, \beta_{\rm disk})+ \#\Gamma_{r}\\
& \leq 2-3\# \Gamma_r + \# \Gamma_r = 2(1-\# \Gamma_r).
\end{aligned}
\end{equation*}
Now since $\#\Gamma_r \geq 1$, we conclude that $\#\Gamma_{r}=1$. Hence $\tilde{u}_{r}$ has only one puncture whose asymptotic limit $P$ is a contractible periodic orbit with $\mu_{CZ}(P, \beta_{\rm disk})=3$.

We have concluded that $\tilde u_r$ is a $\bar J$-holomorphic curve with precisely one negative puncture whose asymptotic limit $P=(x,T)$ is contractible and has Conley-Zehnder index $3$.

Now we want to prove that the vertex $q$, immediately below the root $r$, is associated to an embedded $\tilde J$-holomorphic plane. In particular, $r$ and $q$ are the only vertices of the bubbling-off tree $\mathcal{B}$. Since $\tilde{u}_{r}$ is $\bar{J}$-holomorphic, we conclude that the finite energy sphere $\tilde{u}_{q}=(a_{p}, u_{q}):\mathbb{C}\setminus \Gamma_{q} \rightarrow\mathbb{R}\times L(p, 1)$ is $\tilde{J}$-holomorphic. We know that $\tilde{u}_{q}$ is asymptotic to $P=(x, T)$ at $\infty$, which is a contractible closed Reeb orbit of $\lambda$ with Conley-Zehnder index $3$. Every asymptotic limit of $\tilde u_q$ at a negative puncture is contractible and non-degenerate, with period $\leq \pi r_{1}^{2}$ and Conley-Zehnder index $\geq 3$.

Suppose that $\Gamma_{q}\neq\emptyset$ and denote by $P_{z}$ the asymptotic limit of $\tilde{u}_{q}$ at $z\in\Gamma_{q}$. If the $d\lambda$-area vanishes then $\#\Gamma_{q}\geq2$. Theorem 6.11 in \cite{properties_2} implies that $\tilde{u}_{q}(\mathbb{C} \setminus\Gamma_{q})= \mathbb{R}\times P$.  Denote by $P'$ the prime closed Reeb orbit of $P$, i.e., $P'$ is simple and $P = (P')^{m_\infty}$ for some  $m_\infty\in \N^*$. Then $P_z = (P')^{m_z}$ for some $m_z\in \N^*$, $\forall z\in \Gamma_q$. We have $m_\infty=\sum_z m_z\geq 2$ since $\#\Gamma_q \geq 2$. In particular, $P$ is not simple.

\begin{lemma}\label{lema:potencia_da_orbita}
Let $P=(x,T)\in\mathcal{P}(\lambda)$ be a nondegenerate contractible closed Reeb orbit satisfying $\mu_{CZ}(P,\beta_{\rm disk})=3$ and whose period satisfies $T\leq\pi r_{1}^{2}$. Assume that $P$ is not simple. Then there exists a simple non-contractible closed Reeb orbit $P'\in \P(\lambda)$ so that $P=(P')^{k}$ for some $k\geq 2$. Moreover, $k$ divides $p$ and there exists no $k'\in \{1,\ldots,k-1\}$ such that $(P')^{k'}$ is contractible.
\end{lemma}

Its proof is given in Appendix \ref{appendix_c}. Applying the above lemma to $P$ we conclude that $P'$ is non-contractible and $m_\infty$ is the least positive integer $k$ so that $(P')^k$ is contractible. Since $P_z=(P')^{m_z}$ is contractible and $\#\Gamma_q \geq 2$ we obtain $m_z < m_\infty$ $\forall z\in \Gamma_q$. This is a contradiction and we conclude that the $d\lambda$-area of $\tilde u_q$ is positive.

Now since the $d\lambda$-area of $\tilde u_q$ is positive the invariants ${\rm wind}_\infty(\tilde u_q)$ and ${\rm wind}_\pi(\tilde u_q)$ are well-defined, see section \ref{preliminares}.

Let $\sigma$ be a symplectic trivialization of $\tilde u_q^* \xi$. Since every asymptotic limit  of $\tilde u_q$ is contractible we can assume that $\sigma$ extends over every asymptotic limit as a symplectic trivialization of $\xi$ in class $\beta_{\rm disk}$. Since $\mu_{CZ}(P,\beta_{\rm disk})=3$ and $\mu_{CZ}(P_{z},\beta_{\rm disk})\geq 3, \forall z\in \Gamma_q$, we have
\begin{equation*}
\mathrm{wind}_\infty(\tilde u_q, \infty, \sigma) \leq 1 \mbox{ and } \mathrm{wind}_{\infty}(\tilde{u}_{q}, z,\sigma)\geq2, \forall z \in \Gamma_q.
\end{equation*}
Inequalities above and \eqref{wind_eq} give
\begin{equation*}
\begin{aligned}
0 & \leq \wind_\pi(\tilde u_q) = \wind_\infty(\tilde u_q) -1 + \#\Gamma_q \\
 & = \wind_\infty(\tilde u_q, \infty, \sigma) - \sum_{z\in \Gamma_q} \wind_\infty(\tilde u_q, z, \sigma) -1 + \#\Gamma_q\\
 & \leq
1-2\#\Gamma_{q}-1+\#\Gamma_{q}= -\#\Gamma_{q},
\end{aligned}
\end{equation*}
and, therefore, $\Gamma_{q}=\emptyset$. Thus $\tilde{u}_{q}$ is a $\tilde{J}$-holomorphic finite energy plane and the same identity now implies $\wind_\pi(\tilde u_q)=0$ and $\wind_\infty(\tilde u_q)=1$. Hence
\begin{equation}\label{plano_rapido}
\mathrm{wind}_{\pi}(\tilde{u}_{q})=\mathrm{wind}_{\infty}(\tilde{u}_{q})-1=0.
\end{equation}
In particular, $u_{q}$ is an immersion transverse to the Reeb vector field $X_{\lambda}$. We claim that $\tilde{u}_{q}$ is somewhere injective. Indeed, otherwise we write $\tilde{u}_{q}=\tilde{v}_{q}\circ Q$, for a somewhere injective finite energy $\tilde{J}$-holomorphic plane and a polynomial $Q:\mathbb{C}\rightarrow\mathbb{C}$ with degree ${\rm deg}(Q)\geq 2$. This forces $\tilde u_q$ to have at least one critical point which is impossible since $u_q$ is an immersion.

Now we claim that $\tilde{u}_{q}$ is an embedding. Arguing indirectly, assume that this is not the case. Then
\begin{equation*}
D:=\{(z_{1}, z_{2})\in\mathbb{C} \times\mathbb{C}\setminus \{{\rm diagonal}\}\}: \tilde{u}_{q}(z_{1})=\tilde{u}_{q}(z_{2})\}
\end{equation*}
is non-empty.  Furthermore,  $D$ is discrete, since $\tilde u_q$ is somewhere injective. This follows from the similarity principle. Therefore, self-intersections of $\tilde{u}_{q}$ are isolated. By the positivity and stability of intersections of pseudo-holomorphic curves, it follows that $\tilde{u}_{n}$ also has self-intersections for large $n$. This is a contradiction since every $\tilde u_n$ is embedded.

We have concluded that the bubbling-off tree has exactly two vertices, the root $r$ and a single vertex $q$ below it.  The curve $\tilde u_q$ is an embedded finite energy $\tilde J$-holomorphic plane whose asymptotic limit at $\infty$ is a closed Reeb orbit $P=(x,T)\in \P(\lambda)$ whose Conley-Zehnder index is $3$. Moreover, $\wind_{\pi}(\tilde u_q)=0$ and hence $u_q$ is an immersion transverse to the Reeb vector field $X_{\lambda}$. The winding number of the asymptotic eigensection $e_\infty$ of $\tilde u_q$ at $\infty$ satisfies $\wind_\infty(e_\infty,\beta_{\rm disk}) = 1.$

At this moment we want to prove that
\begin{equation}\label{mergulho_r>>1}
u|_{\mathbb{C}\setminus B_{R}(0)}\hspace{0.2cm}\mathrm{is}\hspace{0.2cm}\mathrm{an}\hspace{0.2cm}\mathrm{embbeding}\hspace{0.2cm}\mathrm{for}\hspace{0.2cm}R\gg1.
\end{equation}

Since $P$ is contractible and $\mu_{CZ}(P,\beta_{\rm disk})=3$, Lemma \ref{lema:potencia_da_orbita} implies that either $P$ is simple or $P=(P')^{k}$, for a simple non-contractible periodic orbit $P'$, where $k\geq 2$ is an integer which divides $p$. If $P$ is simple, \eqref{mergulho_r>>1} follows directly. In the following we assume that $P$ is not simple.

The class $\beta$ of $d\lambda_E$-symplectic trivializations of the contact structure along $P_1$, as given in Lemma \ref{lema:trivializacao}, induces a class $\bar \beta\simeq \beta$ of $d\lambda$-symplectic trivializations of the contact structure along closed curves in class $1\in\Z_p \simeq \pi_1(L(p,1))$. We know that $[P']=q\in\{1, 2, ..., p-1\}$ and that there exists $n\in\mathbb{N}^{*}$ such that
\begin{equation}\label{eq:poten_orbita}
qk=np.
\end{equation}

Take a $d\lambda$-symplectic trivialization of the contact structure along the contractible closed curve $P=(P')^{k}$ in class $\bar{\beta}^{p}$. Since $\mathrm{wind}(\beta_{\mathrm{disk}}, \bar{\beta}^{p})=p-2$, it follows that
\begin{equation*}
\begin{aligned}
\mu_{CZ}((P')^k, \bar{\beta}^{p}) & =\mu_{CZ}((P')^k,\beta_{\mathrm{disk}})+2\mathrm{wind}(\beta_{\mathrm{disk}}, \bar{\beta}^{p})\\
& =3+2(p-2)=2p-1.
\end{aligned}
\end{equation*}
In particular,
\begin{equation}\label{eq:wind_infty}
\mathrm{wind}(e_{\infty}, \bar{\beta}^{p})=p-1.
\end{equation}

Now choose a $d\lambda$-symplectic trivialization of the contact structure along $P'$ in class $\bar{\beta}^{q}$. From \eqref{eq:poten_orbita} and \eqref{eq:wind_infty} we have

\begin{equation}
\begin{aligned}
W:=\mathrm{wind}(e_{\infty},(\bar{\beta}^{q})^{k}) & = \mathrm{wind}(e_\infty, (\bar{\beta}^p)^n)=1+n(p-2).
\end{aligned}
\end{equation}
Putting $h:=\mathrm{gcd}(k, W)$, since $qk=np$, we obtain
\begin{equation}\label{eq:h_embed}
h=\mathrm{gcd}(k, 2n-1).
\end{equation}

We claim that $h=1$. Let $g=\gcd(p,q)$, since $k$ is the least positive integer such that $(P')^k$ is contractible, $k$ is the order of $q$ in $\Z_p$. Hence
\begin{equation}\label{relations:pqkn}
 p=gk,\qquad q=gn,\qquad \gcd(k,n)=1.
\end{equation}

We start by looking for some topological restrictions on $n$. By lifting $\tilde u_q$ to the
universal covering $S^3\to L(p,1)$. Its positive asymptotic orbit is simply covered, otherwise $(P')^j$ would be contractible for some $1\le j<k$. The lifted curve is embedded and has $\operatorname{wind}_\pi=0$. Thus condition (4) in Siefring's embedding criterion  \cite[Theorem 2.6]{siefring2011intersection} is satisfied. Consequently the projection of the lifted plane to $S^3$ is an embedding. The asymptotic compactification is an embedded disk and so the corresponding component $K$ of the preimage of $P'$ is an
unknot.

Realize $L(p,1)$ as the quotient of $S^3\subset\mathbb C^2$ by $G(z_1,z_2)=(\zeta_pz_1,\zeta_pz_2)$. The stabilizer of $K$ is $H=\langle A\rangle\simeq\mathbb Z_k$, where $A=G^g=\zeta_k I$. Let $r$ be the rotation character of $A$ on the oriented circle $K$. A single traversal of $P'$ lifts from $x\in K$ to $G^qx=A^nx$. Hence $nr\equiv1\pmod k$, and therefore $r\equiv n^{-1}\pmod k$. By Freedman--Luo \cite{freedman1995equivariant}, the $H$-invariant unknot $K$ is smoothly $H$-equivariantly isotopic to an $H$-invariant round circle $C$.

If $k>2$, write 
\[ 
A=\cos(2\pi/k)I+\sin(2\pi/k)J.
\]
We have that, if $C=S^3\cap(c+V)$ is round and $A$-invariant, its Euclidean center $c$ is fixed by $A$, hence $c=0$. Moreover $A(V)=V$, and since $\sin(2\pi/k)\ne0$, it follows that $J(V)=V$. Thus $V$ is a complex line and $C$ is a Hopf circle. The generator $A$ acts on $C$ by a rotation through $\pm2\pi/k$, according to orientation. If $k=2$ the same conclusion follows from the antipodal action. Since the isotopy is $H$-equivariant, the rotation character is preserved. Thus $r\equiv\pm1\pmod k$. Since $r\equiv n^{-1}\pmod k$, therefore
\begin{equation}\label{relat_n_mod_k}
 n\equiv\pm1\pmod k.
\end{equation}
For $k>1$, from equation \eqref{relations:pqkn} and $q<p$, we obtain $1\leq n<k$, and conclude that $n=1$ or $n=k-1$.

We use the same component $\Theta'$ fixed above. For the model planes in the ellipsoidal positive end, the projection avoids $P_1$ and the asymptotic winding is extremal, that is, the asymptotic defect satisfies $d_0=0$. Hence Siefring's criterion \cite[Corollary 4.17]{siefring2011intersection}, as used in \cite[Proposition 3.8]{hryniewicz2016elliptic}, gives zero generalized self-intersection for the model planes. By homotopy invariance of Siefring's intersection product,
\begin{equation}\label{eq:siefing}
\tilde u_j*\tilde u_j=0,
\end{equation}
for the sequence $\tilde u_j\in\Theta'$ in Proposition \ref{prop_mininf}.

Let $\mathbf u=\tilde u_r\mid\tilde u_q$ be the two-level SFT limit from Theorem \ref{teo:furo_negativo}. The generalized intersection product extends to holomorphic buildings and is homotopy invariant under continuous deformations in the SFT-topology. From \cite[Proposition 2.11]{cioba2020unknotted}  together with \eqref{eq:siefing}, we obtain
\begin{equation}\label{cioba_and_siefring}
 0=\mathbf u*\mathbf u\ge
 \tilde u_r*\tilde u_r+\tilde u_q*\tilde u_q+k,
\end{equation}
since the unique breaking orbit is $P=(P')^k$, with multiplicity $k$ and
$\mu_{CZ}(P,\beta_{\rm disk})=3$. For the lower plane, the adjunction formula \cite[Theorem 2.3]{siefring2011intersection} satisfies
\begin{equation}\label{adjunction_siefing}
 \tilde u_q*\tilde u_q\ge h-1.
\end{equation}
Indeed, $\tilde u_q$ is a simple index-$2$ plane with one odd puncture, and its
spectral covering number is
$\bar\sigma_-(P)=\gcd(k,W)=h$.

Now suppose that $\tilde u_r$ is somewhere injective. It is an index-$0$ cylinder with two odd punctures. Since each spectral covering number is at least one, the same adjunction formula gives $\tilde u_r*\tilde u_r\ge-1$. From \eqref{cioba_and_siefring} and \eqref{adjunction_siefing} we obtain
$k+h\le2$, hence $k=h=1$.

It remains to consider the case in which $\tilde u_r$ is not somewhere injective. Write $\tilde u_r=\tilde v_r\circ Q$ and $d:=\deg Q>1$. Let $r$ be the covering number of the positive asymptotic orbit of $\tilde v_r$, so that $d=p/r$. By Proposition \ref{lema:ind_2k_1}, $\tilde v_r$ has a unique non-contractible negative end $R$. Write $R=(P')^m$. Since the negative end of $\tilde u_r$ is $P=R^d$, we have
\begin{equation}\label{relation_kmd}
 k=md.
\end{equation}

The cylinder $\tilde v_r$ gives $[R]=r\in H_1(L(p,1);\mathbb Z)$, so $mq\equiv r\pmod p$. Since $d=p/r$, from equations \eqref{relations:pqkn}, \eqref{relation_kmd}, we have $r=gm$. Thus $gmn\equiv gm\pmod{gmd}$, and consequently
\begin{equation}\label{relation_dn}
 d\mid n-1.
\end{equation}
The covering relation for Siefring's intersection product \cite[Proposition C.2]{wendl2017contact}, applied twice and using symmetry, gives $ \tilde u_r*\tilde u_r\ge d^2(\tilde v_r*\tilde v_r)$. The simple cylinder $\tilde v_r$ has index $0$ and two odd punctures, so
adjunction gives $\tilde v_r*\tilde v_r\ge-1$. Hence
\begin{equation}\label{eq:simple_cylinder}
 \tilde u_r*\tilde u_r\ge-d^2.
\end{equation}
Combining \eqref{cioba_and_siefring}, \eqref{adjunction_siefing}, and \eqref{eq:simple_cylinder}, we obtain
\begin{equation}\label{relation_khd}
 k+h-1\le d^2.
\end{equation}

If $n=1$, then \eqref{eq:h_embed} gives $h=1$. Now assume $n=k-1$, then $h=\gcd(k,2k-3)=\gcd(k,3)$. If $h>1$, then $h=3$ and $3$ divides $k$. On the other hand, \eqref{relation_kmd} implies that $d$ divides $k$, whereas \eqref{relation_dn} implies that $d$ divides $k-2$, thus $d$ divides $2$. Since $d>1$, we have $d=2$. By replacing it in \eqref{relation_khd}, we obtain $k+2\le4$. Hence $k\le2$, but it contradicts the fact that $3$ divides $k$. Therefore
\[
 h=\gcd\bigl(k,1+n(p-2)\bigr)=1.
\]
We conclude that $W$ and $k$ are relatively prime numbers. It follows that the puncture of $\tilde u_q$ at $\infty$ is relatively prime and \eqref{mergulho_r>>1} is verified in view of Lemma \ref{mergulho}. Thus $\tilde{u}$ is an embedded fast plane and hence $u(\mathbb{C})\cap x(\mathbb{R})= \emptyset$. For a proof of these facts, see \cite[Section 3.1.8]{hryniewicz2016elliptic}.

Finally, due to \cite[Theorem 2.3]{properties_2},  $u:\mathbb{C} \rightarrow L(p, 1)\setminus x(\mathbb{R})$ is injective. Since $u$ is an immersion it follows that $u$ defines an oriented $k$-disk for $x(\mathbb{R})$, where $k$ is the number of times that $P$ covers $x(\R)$. Writing $P=(P')^k$, where $P'$ is simple, and applying \cite[Lemma 3.10]{HLS}, we obtain ${\rm sl}(P')=-\frac{1}{k}$, where $k$ divides $p$.

Unfortunately we cannot directly apply \cite[Lemma 7.3]{HLS} to obtain a rational open book, since it is stated under global non-degeneracy. Instead, here we make use of its proof from the point where the fast plane has already been obtained, using the compactness result \cite[Theorem 5.3]{HLS}. In our case, every asymptotic orbit which can occur has action bounded by the action of the positive limit. Indeed, by the energy identity its action is at most $T(P)$. Thus, in our construction $
T(P)\leq \pi r_1^2$, and we conclude that all Reeb orbits related with the compactness and continuation argument have action at most $\pi r_1^2$. Hence, the imposed non-degeneracy and dynamical convexity assumptions in Proposition \ref{prop:nao_deg} are enough to carry out the remainder of the proof of \cite[Lemma 7.3]{HLS}. Consequently, we obtain a rational open book decomposition with disk-like pages of order $k$, whose binding is the simple orbit $P'$ and whose pages are rational global surfaces of section.

Applying the characterization theorem of contact lens spaces in \cite{HLS}, we obtain that $L(p,1)$ is diffeomorphic to $L(k,\ell)$ for some $\ell$. Hence
\[
p=|H_1(L(p,1);\mathbb Z)|=|H_1(L(k,\ell);\mathbb Z)|=k.
\]
Therefore $k=p$, the binding is $P'$, and ${\rm sl}(P')=-1/p$. This concludes the proof of Theorem \ref{teo:furo_negativo}. Proposition \ref{prop:nao_deg} also follows.
\end{proof}

\subsection{The degenerate case}

Let $f:L(p, 1)\rightarrow(0, +\infty)$ be the smooth function so that $\lambda=f\lambda_{0}$. Choose $0<r_{1}<r_{2}$, with $r_{1}^{2}/r_{2}^{2}\in\mathbb{R}\setminus\mathbb{Q}$, so that  $f<f_{E}$ pointwise. By \cite[Proposition 6.1]{hofer1998dynamics} we can choose a sequence $f_{n}\rightarrow f$ in $C^{\infty}$ such that $\lambda_{n}:=f_n\lambda_{0}$ is non-degenerate for all $n$ and, moreover, $f_{n}<f_{E}$ pointwise for all large $n$. Now we show that for all large $n$, $\lambda_{n}$ satisfies the hypotheses of Proposition \ref{prop:nao_deg}. If this is not the case then there exists a  subsequence, still denoted by $\lambda_{n}$, so that $\lambda_{n}$ admits a contractible periodic orbit $Q_{n}$ with period $0<T_{n}\leq\pi r_{1}^{2}$ and $\mu_{CZ}(Q_{n})\leq 2$. By the Arzel\`a-Ascoli theorem, up to a subsequence, $Q_{n}$ converges in $C^{\infty}$ as $n\rightarrow\infty$ to a contractible periodic orbit $Q=(w,S)\in \P(\lambda)$ with period $S \leq \pi r_{1}^{2}$. Due to the lower semi-continuity of the generalized Conley-Zehnder index, it follows that $\mu_{CZ}(Q)\leq 2$, which is a contradiction, since $\lambda$ is dynamically convex. Thus, by Proposition \ref{prop:nao_deg}, $\lambda_{n}$ admits a $p$-unknotted periodic orbit $P_{n}=(x_n,T_n)$, with self-linking number $-\frac{1}{p}$ and $\mu_{CZ}(P_{n}^{p})=3$, $\forall n$. Moreover, the period $T_n$ is uniformly bounded by $\pi r_{1}^{2}/p$ in $n$. By the Arzel\`a-Ascoli theorem, up to a subsequence,  $P_{n}\to P$ in $C^{\infty}$ as $n\rightarrow\infty$,  where $P$ is a periodic orbit of $\lambda$ with period $\leq \pi r_{1}^{2}/p$.

Since $P_{n}$ is $p$-unknotted, it follows that $P_{n}^{p}$ is contractible and therefore $P^{p}$ is contractible as well. Due to the lower semi-continuity of the generalized Conley-Zehnder index, we have $\mu_{CZ}(P^{p})\leq3$. Since $\lambda$ is dynamically convex it follows that $\mu_{CZ}(P^{p})=3$. Now suppose by contradiction that $P$ is not simple. Then $P=(P')^{k}$ for some $k\geq2$, where $P'$ is simple. Since $\pi_{1}(L(p, 1))\simeq\mathbb{Z}_{p}$, we have that $(P')^{p}$ is contractible and therefore $\mu_{CZ}((P')^{p},\beta_{\rm disk})\geq 3$. From $P^{p}=((P')^{k})^{p}=((P')^{p})^{k}$, $k\geq2$, it follows that $\mu_{CZ}(P^{p},\beta_{\rm disk})\geq 5$, which is a contradiction. Hence $P$ is simple and $P_{n}\rightarrow P$ as $n\rightarrow\infty$. The simple periodic orbit $P$ is transversely isotopic to each $P_{n}$ for all large $n$. Hence $P$ is $p$-unknotted and has self-linking number $-\frac{1}{p}$. Applying  \cite[Corollary 1.8]{hryniewicz2016elliptic}, the periodic orbit $P$ bounds a $p$-disk which is a rational global surface of section for the Reeb flow. Moreover, this $p$-disk is a page of a rational open book decomposition of $L(p, 1)$ with binding $P$ such that all pages are rational global surfaces of section. The proof of Theorem \ref{teo:lp} is now complete.

\appendix

\section{Proof of Lemma \ref{lema:unico_furo}}\label{appendix_a}

Recall that $P_{\infty}=P_1^k$ is non-contractible since $1\leq k<p$. Since $\tilde v_r$ is a punctured sphere, the homology classes of its asymptotic limits satisfy
\[
[P_\infty]=\sum_{w\in\Gamma'}[P_w]\neq0\quad\mbox{ in }\mathbb Z_p.
\]
Hence there exists $w_1\in\Gamma'$ such that $P_{w_1}$ is non-contractible. Every point $z_i\in Q^{-1}(w_1)$ has multiplicity at least two, otherwise, if one of these pre-images were simple, then $Q$ would be a local biholomorphism near $z_i$, and the asymptotic limit of $\tilde u_r=\tilde v_r\circ Q$ at $z_i$ would coincide with $P_{w_1}$. This is impossible, since every negative asymptotic limit of $\tilde u_r$ is contractible.

Write
\begin{equation}\label{pol}
Q(z)=w_{1}+a(z-z_{1})^{k_{1}}(z-z_{2})^{k_{2}} \ldots (z-z_{n})^{k_{n}},
\end{equation} where $0\neq a,z_i\in \C$, $k_i\in \N^*$ and $z_i\neq z_j, \forall i\neq j$. Notice that $Q(z_i)=w_1, \forall i=1\ldots n$.
We claim that
\begin{equation}\label{multiplicidade}
k_{i}\geq2,  \ \ \forall i=1, ..., n.
\end{equation}
By the choice of $w_1$ above, \eqref{multiplicidade} holds and, in particular,
\begin{equation}\label{num_pon_crit}
n\leq \frac{{\rm deg}(Q)}{2}.
\end{equation} Moreover, each $z_i$ is a critical point of $Q$ with multiplicity $k_i-1\geq 1$. Hence the total contribution of $z_i$ to the number of critical points of $Q$ is given by
\begin{equation}\label{tot1}
\sum_{i=1}^n (k_i-1) = {\rm deg}(Q) - n \geq \frac{{\rm deg}(Q)}{2}.
\end{equation}

Now suppose that there exists a puncture $w_2\neq w_1 \in\Gamma'$ of $\tilde v_r$ whose asymptotic limit is also  non-contractible. As before, write
\begin{equation}\label{pol_2}
Q(z)=w_2+b(z-\bar{z}_{1})^{l_{1}}(z-\bar{z}_{2})^{l_{2}} \ldots (z-\bar{z}_{m})^{l_{m}},
\end{equation}
where $0\neq b, \bar z_i \in \C$ is a constant, $l_{i}\geq 2$, $\forall j=1, ..., m$, and
\begin{equation}\label{num_p_crit_2}
m\leq \frac{{\rm deg}(Q)}{2}.
\end{equation}

Since $w_2 \neq w_1$, we clearly have $\bar z_j \neq z_i$, $\forall i,j$. Hence the total contribution of $\bar z_i$ to the critical points of $Q$ is given by
\begin{equation}\label{tot2}
\sum_{i=1}^m (l_i-1) = {\rm deg}(Q) - m \geq \frac{{\rm deg}(Q)}{2}.
\end{equation}
Summing up \eqref{tot1} and \eqref{tot2} we conclude that $Q$ has at least ${\rm deg}(Q)$ critical points, a contradiciton. We conclude that $\tilde v_r$ has precisely one puncture whose asymptotic limit is non-contractible.
\qed

\section{Proof of Lemma \ref{lema:trivializacao}}\label{appendix_b}

Let us consider the $\lambda_E$-closed Reeb orbit $\tilde P_1=(x_1,\tilde T_1=\pi r_1^2)\subset S^{3}$, where $x_1(t)= (e^{2\pi it/\tilde T_1}, 0)\in \C^2$, $t\in\R / \tilde T_1\Z$. The contact structure $\xi_0$ along $P_1$ is spanned by $\{\partial_{x_2},\partial_{y_2}\}$ in coordinates $(x_1+iy_1,x_2+iy_2)\in \C^2$. Consider the non-vanishing section $\tilde Z_1$ of $x_1(\tilde T_1\cdot) ^*\xi_0$ given  by $\tilde Z_1(t)=-e^{-2\pi it}$, $t\in\mathbb{R}/\mathbb{Z}$, where $$a+ib \equiv a\partial_{x_2} + b\partial_{y_2}.$$ Defining $\tilde Z_2 = i\tilde Z_1$ we see that $\{\tilde Z_1,\tilde Z_2\}$ induces the class $\beta_{\rm disk}$ of $d\lambda_E$-symplectic trivializations of $x_1(\tilde T_1\cdot)^*\xi_0$, see \cite{coloquio_finsler}.

Now consider the following non-vanishing section of $x_1(\tilde T_1\cdot)^*\xi_0$
\begin{equation}\label{def:secao_n_nula}
\tilde{Z}(t):=-e^{(p-1)(-2\pi it)},\ \ t\in\R/ \Z.
\end{equation}
 A direct computation shows that ${\rm wind}(\tilde Z, \tilde Z_1)=2-p$.
Hence, denoting by $\tilde \beta$ the homotopy class of the $d\lambda_E$-symplectic trivializations of $x_1(\tilde T_1\cdot)^*\xi_0$ induced by $\tilde{Z}$, we  obtain
\begin{equation}\label{windbeta}
\mathrm{wind}(\tilde{\beta}, \beta_{\rm disk})=2-p.
\end{equation}
Furthermore,  the section $\tilde{Z}$ presents the following $\Z_p$-symmetry
\begin{equation}\label{simetria}
\tilde{Z}(t+ 1/p)=e^{2\pi i\frac{1}{p}}\tilde{Z}(t), \forall t\in \R / \Z,
\end{equation}
see Figure \ref{fig:trivia_ilust}. Hence, using the projection $\pi_{p, 1}:S^{3}\rightarrow L(p, 1)$, we see that $\tilde Z$ descends to a non-vanishing section $\bar Z$ of the contact structure $\xi_{\rm std}$ along the simple closed Reeb orbit $P_1=(x_1,T_1=\tilde T_1/p)\subset L(p,1)$.

Denote by $\beta$ the homotopy class of $d\lambda_E$-symplectic trivializations of $x_1(T_1\cdot)^*\xi_0$ induced by $\bar Z$. It follows from \eqref{windbeta} that
\begin{equation}\label{wind_2}
{\rm wind}(\beta^p, \beta_{\rm disk})=2-p.
\end{equation}

We know that $\mu_{CZ}(P_1^{p}, \beta_{\rm disk})=3$ and hence, by (\ref{relacao_conley_bases}) and \eqref{wind_2}, we obtain
\begin{equation}\label{eq:indice_2p_1}
\mu_{CZ}(P_1^p, \beta^p)=\mu_{CZ}(P_1^p,\beta_{\rm disk})-2\mathrm{wind}(\beta^p, \beta_{\rm disk})
=3-2(2-p)=2p-1.
\end{equation}

Since $P_1$ is elliptic and nondegenerate, we know from (\ref{relacao_conley_rotacao}) that
\begin{equation}\label{eq:ind_dado_rodacao}
2p-1=\mu_{CZ}(P_1^p, \beta^p)=2\lfloor p\rho(P_1, \beta)\rfloor+1,
\end{equation}
which implies that
\begin{equation*}
\frac{p-1}{p}<\rho(P_{1}, \beta)<1.
\end{equation*}
Now given $1\leq k\leq p$, $k\in\mathbb{N}$, we obtain from the inequalities above that
\begin{equation*}
k-1\leq k\frac{(p-1)}{p}<k\rho(P_{1}, \beta)<k,
\end{equation*}
and hence $\lfloor k\rho(P_{1}, \beta)\rfloor=k-1.$  This implies that $\mu_{CZ}(P_1^k, \beta^k)=2k-1$.
\qed

\begin{figure}
    \includegraphics[width=0.75\textwidth]{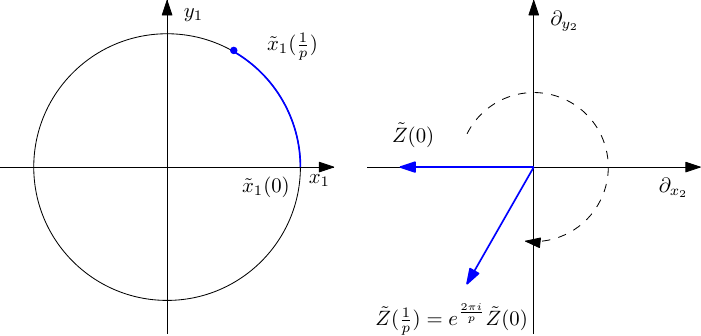}
    \caption{The non-vanishing section $\tilde{Z}$ is $\Z_p$-symmetric.}
    \label{fig:trivia_ilust}
\end{figure}

\section{Proof of Lemma \ref{lema:potencia_da_orbita}}\label{appendix_c}

Let $P'$ be a  contractible periodic orbit so that $P=(P')^{k}$. Since $P$ is nondegenerate, $P'$ is nondegenerate as well. Using \eqref{eq:numero_rot}, \eqref{relacao_conley_rotacao} and that $\mu_{CZ}(P,\beta_{\rm disk})=3$ we obtain
\begin{equation}\label{ineq}
1<\rho(P,\beta_{\rm disk})<2 \Rightarrow \frac{1}{k}<\rho(P',\beta_{\rm disk})<\frac{2}{k}.
\end{equation}
 Since the period of $P'$ is $T/k \leq \pi r_1^2$,  the assumptions on $\lambda$ imply that
\begin{equation*}
\mu_{CZ}(P',\beta_{\rm disk})\geq 3.
\end{equation*}
Hence $\rho(P',\beta_{\rm disk})>1$ which, in view of \eqref{ineq}, gives $k=1$.

We conclude that if $P$ is not simple then there exist $k\geq 2$ and a simple non-contractible closed orbit $P'\in \P(\lambda)$ so that $(P')^k = P$. Next we prove that such $k$ divides $p$. Since $\pi_{1}(L(p, 1))\simeq\mathbb{Z}_{p}$ and $P'$ is non-contractible, we have
\begin{equation*}
[P']=q\in\{1, 2, ..., p-1\}.
\end{equation*}
Since $P$ is contractible, $qk=np$ for some $n\in\mathbb{N}^*$.

If $k$ does not divide $p$ then $k=k_1k_2$ for some integers $k_1,k_2$ where $k_1\geq 2$ divides $n$. Then  $(P')^{k_2}$ is contractible since $qk_2 = \frac{n}{k_1}p$ is a multiple of $p$. Using the above argument we must have $k_1=1$ since $((P')^{k_2})^{k_1}=P$. This is a contradiction and we conclude that $k$ divides $p$.

Now assume that there exists $k'\in \{1,\ldots,k-1\}$ so that $(P')^{k'}$ is contractible. Take the least $k'$ with this property. We claim that $k'$ divides $k$. If this is not the case then let $1\leq r<k'$ be the remainder of the division $k/k'$ and let $m\in \N^*$ be such that $k=mk'+r$. Since $(P')^{k'}$ is contractible we have $qk'=n'p$ for some $n'\in \N^*$. Then $qr = qk - qmk' = np-mn'p$ is a multiple of $p$, which implies that $(P')^r$ is contractible. This contradicts the minimality of $k'$. We conclude that $k'$ divides $k$. Since $(P')^{k'}$ is contractible and $((P')^{k'})^{k/k'}=P$ we obtain by the previous argument that $k/k'=1$, again a contradiction. We conclude that there exists no $k'\in \{1,\ldots,k-1\}$ so that $(P')^{k'}$ is contractible.
\qed

\bibliographystyle{plain}
\bibliography{bibliografia}

\begin{thebibliography}{10}

\bibitem{albers2012global}
P.~Albers, J.~W. Fish, U.~Frauenfelder, H.~Hofer, and O.~van Koert.
\newblock {Global surfaces of section in the planar restricted 3-body problem}.
\newblock {\em Archive for Rational Mechanics and Analysis}, 204(1):273--284,
  2012.

\bibitem{albers2012contact}
P.~Albers, U.~Frauenfelder, O.~{Van Koert}, and G.~P. Paternain.
\newblock {Contact geometry of the restricted three-body problem}.
\newblock {\em Communications on Pure and Applied Mathematics}, 65(2):229--263,
  2012.

\bibitem{arioli2001symbolic}
G.~Arioli and P.~Zgliczy{\'n}ski.
\newblock {Symbolic dynamics for the H{\'e}non--Heiles Hamiltonian on the
  critical level}.
\newblock {\em Journal of Differential Equations}, 171(1):173--202, 2001.

\bibitem{baker2012rational}
K.~Baker and J.~Etnyre.
\newblock {Rational linking and contact geometry}.
\newblock In {\em {Perspectives in analysis, geometry, and topology}}, pages
  19--37. Springer, 2012.

\bibitem{bourgeois2003compactness}
F.~Bourgeois, Y.~Eliashberg, H.~Hofer, K.~Wysocki, and E.~Zehnder.
\newblock {Compactness results in symplectic field theory}.
\newblock {\em Geometry \& Topology}, 7(2):799--888, 2003.

\bibitem{churchill1979survey}
R.~C. Churchill, G.~Pecelli, and D.~L. Rod.
\newblock {A survey of the H{\'e}non-Heiles Hamiltonian with applications to
  related examples}.
\newblock {\em Stochastic behavior in classical and quantum Hamiltonian
  systems}, pages 76--136, 1979.

\bibitem{churchill1980pathology}
R.~C. Churchill and D.~L. Rod.
\newblock {Pathology in dynamical systems. III. Analytic Hamiltonians}.
\newblock {\em Journal of Differential Equations}, 37(1):23--38, 1980.

\bibitem{cioba2020unknotted}
A.~Cioba and C.~Wendl.
\newblock Unknotted reeb orbits and nicely embedded holomorphic curves.
\newblock {\em Journal of Symplectic Geometry}, 18(1):57--122, 2020.

\bibitem{dePaulo_Salomao2}
N.~de~Paulo and Pedro A.~S. Salom{\~a}o.
\newblock {On the multiplicity of periodic orbits and homoclinics near critical
  energy levels of Hamiltonian systems in $\mathbb{R}^4$}.
\newblock {\em Transactions of the American Mathematical Society}, 2018.

\bibitem{dePaulo_Salomao}
N.~de~Paulo and Pedro A.~S. Salom{\~a}o.
\newblock {Systems of transversal sections near critical energy levels of
  Hamiltonian systems in $\mathbb{R}^4$}.
\newblock {\em Memoirs of the American Mathematical Society}, 252(1202):105pp,
  2018.

\bibitem{franks1992geodesics}
J.~Franks.
\newblock {Geodesics on $S^2$ and periodic points of annulus homeomorphisms}.
\newblock {\em Inventiones mathematicae}, 108(1):403--418, 1992.

\bibitem{frauenfelder2016real}
U.~Frauenfelder and J.~Kang.
\newblock {Real holomorphic curves and invariant global surfaces of section}.
\newblock {\em Proceedings of the London Mathematical Society},
  112(3):477--511, 2016.

\bibitem{freedman1995equivariant}
M.~H. Freedman and F.~Luo.
\newblock Equivariant isotopy of unknots to round circles.
\newblock {\em Topology and its Applications}, 64(1):59--74, 1995.

\bibitem{ghptorsioncontact2017}
D.~C. Gardiner, M.~Hutchings, and D.~Pomerleano.
\newblock {Torsion contact forms in three dimensions have two or infinitely
  many Reeb orbits}.
\newblock {\em arXiv:1701.02262}, 2017.

\bibitem{henon1964applicability}
M.~Henon and C.~Heiles.
\newblock {The applicability of the third integral of motion: some numerical
  experiments}.
\newblock {\em The Astronomical Journal}, 69:73, 1964.

\bibitem{hofer1993pseudoholomorphic}
H.~Hofer.
\newblock {Pseudoholomorphic curves in symplectizations with applications to
  the Weinstein conjecture in dimension three}.
\newblock {\em Inventiones mathematicae}, 114(1):515--563, 1993.

\bibitem{properties_2}
H.~Hofer, K.~Wysocki, and E.~Zehnder.
\newblock {Properties of pseudo-holomorphic curves in symplectisations II:
  Embedding controls and algebraic invariants}.
\newblock In {\em {Geometries in Interaction}}, pages 270--328. Springer, 1995.

\bibitem{hofer1996characterisation}
H.~Hofer, K.~Wysocki, and E.~Zehnder.
\newblock {A characterisation of the tight three-sphere}.
\newblock {\em Duke J. Math}, 81(1):159--226, 1996.

\bibitem{properties_1}
H.~Hofer, K.~Wysocki, and E.~Zehnder.
\newblock {Properties of pseudoholomorphic curves in symplectisations I :
  asymptotics}.
\newblock {\em Annales de l'I.H.P. Analyse non lin{\'e}aire}, 13(3):337--379,
  1996.

\bibitem{hofer1998dynamics}
H.~Hofer, K.~Wysocki, and E.~Zehnder.
\newblock {The dynamics on three-dimensional strictly convex energy surfaces}.
\newblock {\em Annals of Mathematics}, pages 197--289, 1998.

\bibitem{hhofer1999characterization}
H.~Hofer, K.~Wysocki, and E.~Zehnder.
\newblock {A characterization of the tight three sphere II}.
\newblock {\em Commun. Pure Appl. Anal}, 55:1139--1177, 1999.

\bibitem{properties_3}
H.~Hofer, K.~Wysocki, and E.~Zehnder.
\newblock {\em {Properties of Pseudoholomorphic Curves in Symplectizations III:
  Fredholm Theory}}, pages 381--475.
\newblock Birkh{\"a}user Basel, Basel, 1999.

\bibitem{hofer2003finite}
H.~Hofer, K.~Wysocki, and E.~Zehnder.
\newblock {Finite energy foliations of tight three-spheres and Hamiltonian
  dynamics}.
\newblock {\em Annals of Mathematics}, pages 125--255, 2003.

\bibitem{hryniewicz2012fast}
U.~L. Hryniewicz.
\newblock {Fast finite-energy planes in symplectizations and applications}.
\newblock {\em Transactions of the American Mathematical Society},
  364(4):1859--1931, 2012.

\bibitem{HLS}
U.~L. Hryniewicz, J.~E. Licata, and P.~A.~S. Salom{\~a}o.
\newblock {A dynamical characterization of universally tight lens spaces}.
\newblock {\em Proceedings of the London Mathematical Society},
  110(1):213--269, 2015.

\bibitem{HMS}
U.~L. Hryniewicz, A.~Momin, and P.~A.~S. Salom{\~a}o.
\newblock {A Poincar{\'e}--Birkhoff theorem for tight Reeb flows on $S^3$}.
\newblock {\em Inventiones mathematicae}, 199(2):333--422, 2015.

\bibitem{pedro_proceedings_icm}
U.~L. Hryniewicz and P.~A.~S. Salom\~ao.
\newblock Global surfaces of section for reeb flows in dimension three and
  beyond.
\newblock {\em Proceedings of the ICM 2018}, 1:937--964, 2018.

\bibitem{coloquio_contato}
U.~L. Hryniewicz and P.~A.~S. Salom{\~a}o.
\newblock {\em {Uma introdu\c{c}{\~a}o {\`a} geometria de contato e
  aplica\c{c}\~oes {\`a} din{\^a}mica Hamiltoniana.}}
\newblock Publica\c{c}\~oes Matem{\'a}ticas do IMPA, 2009.

\bibitem{hs_ontheexistenceofdisk2011}
U.~L. Hryniewicz and P.~A.~S. Salom{\~a}o.
\newblock {On the existence of disk-like global sections for Reeb flows on the
  tight 3-sphere}.
\newblock {\em Duke J. Math}, 160(3):415--465, 2011.

\bibitem{coloquio_finsler}
U.~L. Hryniewicz and P.~A.~S. Salom{\~a}o.
\newblock {\em {Introdu\c{c}{\~a}o {\`a} Geometria Finsler}}.
\newblock Publica\c{c}\~oes Matem{\'a}ticas do IMPA, 1 edition, 2013.

\bibitem{hryniewicz2016elliptic}
U.~L. Hryniewicz and P.~A.~S. Salom{\~a}o.
\newblock {Elliptic bindings for dynamically convex Reeb flows on the real
  projective three-space}.
\newblock {\em Calculus of Variations and Partial Differential Equations},
  55(2):1--57, 2016.

\bibitem{ragazzo1994nonintegrability}
C.~G. Ragazzo.
\newblock {Nonintegrability of some Hamiltonian systems, scattering and
  analytic continuation}.
\newblock {\em Communications in Mathematical Physics}, 166(2):255--277, 1994.

\bibitem{salomao2004convex}
P.~A.~S. Salom\~ao.
\newblock {Convex energy levels of Hamiltonian systems}.
\newblock {\em Qualitative Theory of Dynamical Systems}, 4(2):439--454, 2004.

\bibitem{siefring2008relative}
R.~Siefring.
\newblock {Relative asymptotic behavior of pseudoholomorphic half-cylinders}.
\newblock {\em Communications on Pure and Applied Mathematics},
  61(12):1631--1684, 2008.

\bibitem{siefring2011intersection}
R.~Siefring.
\newblock {Intersection theory of punctured pseudoholomorphic curves}.
\newblock {\em Geometry \& Topology}, 15(4):2351--2457, 2011.

\bibitem{wendl2017contact}
C.~Wendl.
\newblock Contact 3-manifolds, holomorphic curves and intersection theory,
  2019.

\end{thebibliography}

\end{document}